\documentclass[11pt]{amsart}

\usepackage{appendix} 

\usepackage{tabularx, environ}

\usepackage{amsmath, amsthm, amssymb, amsfonts} 

\allowdisplaybreaks 

\usepackage{eurosym} 

\usepackage{xcolor} 

\usepackage{stmaryrd} 
\usepackage{mathabx} 
\usepackage{bm} 

\usepackage[hidelinks]{hyperref} 
\usepackage{color} 
\usepackage[normalem]{ulem} 
\usepackage{comment} 

\usepackage{graphicx} 
\usepackage{wrapfig} 
\usepackage{caption} 
\usepackage{subcaption} 
\usepackage{multirow} 
\usepackage{array} 
\usepackage[section]{placeins} 

\usepackage{tikz} 
\usepackage{pgfplots} 
\pgfplotsset{compat=1.12} 

\newtheorem{thm}{Theorem}[section]
\newtheorem{cor}[thm]{Corollary}
\newtheorem{lem}[thm]{Lemma}
\newtheorem{prop}[thm]{Proposition}

\theoremstyle{definition}
\newtheorem{defn}[thm]{Definition}
\newtheorem{ass}[thm]{Assumption}

\theoremstyle{remark}
\newtheorem{rem}[thm]{Remark}

\newcommand{\R}{\mathbb{R}}
\newcommand{\C}{\mathbb{C}}
\newcommand{\N}{\mathbb{N}}

\newcommand{\F}{\mathbb{F}}

\newcommand{\ind}{\bm{1}}

\newcommand{\ud}{\mathrm{d}}

\newcommand{\im}{\ensuremath{\mathsf{i}}}

\renewcommand{\Re}{\mathrm{Re}}
\renewcommand{\Im}{\mathrm{Im}}

\newcommand{\be}{\begin{equation}}
\newcommand{\ee}{\end{equation}}
\newcommand{\ba}{\begin{aligned}}
\newcommand{\ea}{\end{aligned}}

\numberwithin{equation}{section}

\newcommand{\cC}{\mathcal{C}}

\newcommand{\cD}{\mathcal{D}}
\newcommand{\cE}{\mathcal{E}}
\newcommand{\cF}{\mathcal{F}}
\newcommand{\cG}{\mathcal{G}}

\newcommand{\cK}{\mathcal{K}}

\newcommand{\cS}{\mathcal{S}}
\newcommand{\cU}{\mathcal{U}}
\newcommand{\cV}{\mathcal{V}}
\newcommand{\cW}{\mathcal{W}}
\newcommand{\cX}{\mathcal{X}}

\newcommand{\FF}{\mathbb{F}}
\newcommand{\GG}{\mathbb{G}}
\newcommand{\EE}{\mathbb{E}}
\newcommand{\PP}{\mathbb{P}}

\newcommand{\tildeN}{\widetilde{N}}
\newcommand{\uT}{\mathrm{T}}

\newcommand{\uCBITCL}{\mathrm{CBITCL}}

\newcommand{\horizon}{\mathcal{T}} 

\newcommand{\dbra}[1]{[\kern-0.15em[ #1 ]\kern-0.15em]}
\newcommand{\dbraco}[1]{[\kern-0.15em[ #1 [\kern-0.15em[}
\newcommand{\dbraoc}[1]{]\kern-0.15em] #1 ]\kern-0.15em]}
\newcommand{\dbraoo}[1]{]\kern-0.15em] #1 [\kern-0.15em[}

\title{CBI-time-changed L\'evy processes}

\author[C. Fontana]{Claudio Fontana}
\address{Department of Mathematics ``Tullio Levi Civita'', University of Padova (Italy)}
\email{fontana@math.unipd.it}

\author[A. Gnoatto]{Alessandro Gnoatto}
\address{Department of Economics, University of Verona (Italy)}
\email{alessandro.gnoatto@univr.it}

\author[G. Szulda]{Guillaume Szulda}
\address{CERMICS, Ecole des Ponts, Marne-la-Vall\'ee (France).}
\email{guillaume.szulda@enpc.fr}

\thanks{\emph{Acknowledgements:} The authors are thankful to two anonymous Reviewers for their constructive remarks which helped to improve the paper. C.F. is grateful to the Europlace Institute of Finance for financial support to this work. G.S. acknowledges hospitality and financial support from the University of Verona, where part of this work has been conducted. This work is part of the project BIRD 2019 ``Term structure dynamics in interest rate and energy markets: modelling and numerics'' and of the project STARS@UNIPD - PRISMA (Probabilistic Methods for Information in Security Markets), both funded by the University of Padova.}

\date{\today}

\begin{document}

\begin{abstract}
We introduce and study the class of {\em CBI-time-changed L\'evy processes} (CBITCL), obtained by time-changing a L\'evy process with respect to an integrated continuous-state branching process with immigration (CBI). We characterize CBITCL processes as solutions to a certain stochastic integral equation and relate them to affine stochastic volatility processes. We provide a complete analysis of the time of explosion of exponential moments of CBITCL processes and study their asymptotic behavior. In addition, we show that CBITCL processes are stable with respect to a suitable class of equivalent changes of measure. As illustrated by some examples, CBITCL processes are flexible and tractable processes with a significant potential for applications in finance.
\end{abstract}

\maketitle

\section{Introduction}
\label{sec:introcbitcl}

Since their introduction in \cite{KW71}, continuous-state branching processes with immigration (CBI processes) have represented a major topic of research in the theory of stochastic processes (we refer to \cite{Li20} for a recent overview of some of the main developments in the field). 
CBI processes belong to the class of affine processes (see \cite{DFS03}) and, due to their analytical tractability, have found important applications in mathematical finance, especially in interest rate modelling (see \cite{Fil01}).\\
In recent years, CBI processes have attracted a renewed interest in financial modelling, due to their capability of reproducing empirical features of financial time series such as volatility clustering and self-exciting jumps. 
In particular, self-exciting CBI processes have been exploited for the construction of single-curve and multi-curve interest rate models in \cite{JMS17,FGS21}, for the modelling of energy prices in \cite{JMSS19,CMS21} and for stochastic volatility modelling in \cite{JMSZ21}.

In this paper, we introduce and study the class of {\em CBI-time-changed L\'evy processes} (CBITCL), obtained by time-changing a L\'evy process with respect to the time integral of a CBI process. This construction combines the distributional flexibility of L\'evy processes with the self-exciting behavior of CBI processes, while retaining full analytical tractability. CBITCL processes have a significant potential for applications in finance, notably in markets with stochastic volatility, as illustrated in the companion paper \cite{FGS22}. In this work, we lay the theoretical foundations of CBITCL processes and derive some results that are especially motivated by financial applications.

The main contributions of the present paper can be outlined as follows:
\begin{itemize}
\item We characterize CBITCL processes as solutions to a stochastic integral equation which generalizes the well-known Dawson-Li representation of a CBI process (see \cite{DL06}) and we provide an additional characterization in terms of their semimartingale characteristics. Moreover, we study the relation with affine stochastic volatility processes, adopting an extension of the original definition of \cite{KR11}, and derive necessary and sufficient conditions for a general affine stochastic volatility process to be a CBITCL process.
\item By combining techniques of affine processes and specific properties of CBITCL processes, we analyse the time of explosion of exponential moments of CBITCL processes, under a suitable assumption on their dependence structure. This has important consequences in finance, where a CBITCL process can be used to represent the log-price of an asset. We also study the asymptotic behavior of CBITCL processes. While the existence of a stationary distribution of a CBI process is well understood (see, e.g., \cite[Section 10]{Li20}), we provide an asymptotic result analogous to \cite[Theorem 3.4]{KR11} for the distribution of a CBITCL process, making use of our characterization of the lifetime of exponential moments. 
\item In view of financial applications, we derive a class of equivalent changes of probability that leave invariant the class of CBITCL processes, up to a change in their parameters. Moreover, we illustrate how our results can be applied to some specifications of CBITCL processes that have been recently considered in mathematical finance.
\end{itemize}
We emphasize that, while some of our results can be derived from the general theory of affine processes, we obtain more refined statements under minimal technical assumptions by exploiting the specific structure of CBITCL processes.

Our work is naturally related to the use of stochastic changes of time in finance. Starting from the seminal work \cite{Clark73}, time-changed processes have been widely adopted as models for asset prices and we refer to \cite{BNS15,S16} for detailed accounts on the topic. In particular, our work builds on the contributions of \cite{CGMY03,CW04}, where stochastic volatility models have been constructed by relying on time-changed L\'evy processes. An empirical analysis of several specifications of such models has been conducted in \cite{HW04}. While their analysis only considers changes of time driven by square-root diffusions, \cite{HW04} point out that a promising direction of research is the study of models where the activity rate of the time change process exhibits high-frequency jumps. This has been confirmed by the empirical analysis conducted in \cite{FHM21}.
Our work contributes to this line of research by developing a general theoretical framework for the use of CBI processes as changes of time for L\'evy processes, also allowing for the presence of self-exciting jumps in the activity rate of the time change and for dependence between the L\'evy process and the activity rate.

{\bf Structure of the paper.}
In Section \ref{sec:defcbitcl}, we state the definition of CBITCL processes, characterize them as solutions to a stochastic integral equation and study their relation to affine stochastic volatility processes. In Section \ref{sec:moments}, under a suitable assumption on the dependence structure of a CBITCL process, we derive an explicit characterization of the lifetime of exponential moments and of the asymptotic behavior of CBITCL processes. In Section \ref{sec:girsanovcbitcl}, we present a class of equivalent changes of probability that leave invariant the class of CBITCL processes. Finally, in Section \ref{sec:examples} we present two examples of CBITCL processes that are relevant for financial applications and discuss the applications of CBITCL processes to option pricing.

{\bf Notation.} In the following, we denote by $\R_+$ ($\R_-$, resp.) the set of non-negative (non-positive, resp.) real numbers. We denote by $\C_+$ ($\C_-$, resp.) the set of complex numbers with non-negative (non-positive, resp.) real part. The set $\im\R$ is the set of complex numbers with null real part.

\section{Definition and characterization}\label{sec:defcbitcl}

In this section, we state the definition of a CBITCL process and prove some foundational results for this class of processes. We let $(\Omega,\cF,\PP)$ be a probability space supporting the following objects:
\begin{itemize}
\item[(i)] a one-dimensional L\'evy process $K=(K_t)_{t\geq0}$ with triplet $(\beta, 0, \nu)$, where $\beta\geq0$ and $\nu$ is a L\'evy measure\footnote{We recall that a measure $\gamma$ on $\R^n$ is a L\'evy measure if $\gamma(\{0\})=0$ and $\int_{\R^n}(|x|^2\wedge 1)\gamma(\ud x)<+\infty$.} on $\R_+$ such that $\int_0^1 {x\,\nu(\ud x)}<+\infty$. 
The exponent $\Psi:\C_-\to\C$ is given by
\be\label{eq:immigrationmechanism}
\Psi(u) = \beta u + \int_{\R_+} {(e^{u x} - 1)\nu(\ud x)}, 
\qquad \text{for all } u \in\C_-;
\ee	
\item[(ii)] a two-dimensional L\'evy process $(M, N) =((M_t, N_t))_{t\geq0}$, independent of $K$, with triplet $(b, \sigma^2, \Pi)$, where $\Pi(\ud x, \ud z)$ is a L\'evy measure on $\R_+\times\R$ and 
\be\label{eq:triplet}
b = \begin{pmatrix}
-b_X \\
b_Z
\end{pmatrix}, \qquad
\sigma^2 = \begin{pmatrix}
\sigma_X^2 & \rho\,\sigma_X\sigma_Z \\
\rho\,\sigma_X\sigma_Z & \sigma_Z^2
\end{pmatrix}, \qquad
\int_{[1, +\infty)\times\R} {x\,\Pi(\ud x, \ud z) < +\infty},
\ee
with $b_X, b_Z\in\R$, $\sigma_X, \sigma_Z\in\R_+$ and $\rho\in[-1, 1]$. The exponent $\Lambda:\C_-\times\im\R\rightarrow\C$ is given by 
\be	\label{eq:levyexp}\begin{aligned}
\Lambda(u_1, u_2) &= -b_X u_1 + b_Z u_2 + \frac{1}{2}\sigma_X^2u_1^2 + \rho\,\sigma_X\sigma_Z\,u_1 u_2 + \frac{1}{2}\sigma_Z^2u_2^2\\
&\quad+\int_{\R_+\times\R}{\bigl(e^{u_1 x + u_2 z} - 1 - u_1 x - u_2z\ind_{\{|z|<1\}}\bigr)\Pi(\ud x, \ud z)},
\end{aligned}\ee
for all $(u_1, u_2)\in\C_-\times\im\R$.
The process $M$ is therefore a spectrally positive L\'evy process with finite first moment and  exponent $\Phi:\C_-\to\C$ given by 
\be\label{eq:branchingmechanism}
\Phi(u) = -b_X u +  \frac{1}{2}\sigma_X^2u^2 + \int_{\R_+\times\R}{(e^{ux} - 1 - ux)\Pi(\ud x, \ud z)}, \qquad \text{for all } u \in\C_-,
\ee
whereas $N$ is a general L\'evy process whose exponent $\Xi:\im\R\to\C$ is given by
\be\label{eq:xi}
\Xi(u) = b_Z u + \frac{1}{2}\sigma_Z^2u^2 + \int_{\R_+\times\R}{(e^{uz} - 1 - uz\ind_{\{|z|<1\}})\Pi(\ud x, \ud z)}, 
\qquad \text{for all } u \in \im\R.
\ee
\end{itemize}

\begin{defn}\label{def:cbitcl}
A process $(X,Z)=((X_t, Z_t))_{t\geq0}$ taking values in the state space $\R_+\times\R$ is a \emph{CBI-time-changed L\'evy process} (CBITCL) if it holds that
\be\label{eq:stochtimechangeequation}
\begin{cases}
X_t = X_0 + M_{Y_t} + K_t,\\
Z_t = N_{Y_t},
\end{cases}
\ee
for all $t \geq 0$, with $X_0\geq0$ and where $Y=(Y_t)_{t\geq0}$ denotes the process $Y:=\int_0^{\cdot} {X_s\,\ud s}$.
\end{defn}

The above definition is well-posed due to the existence of a unique strong solution $(X,Z)$ to the stochastic time change equation \eqref{eq:stochtimechangeequation}, which follows as a special case of \cite[Theorem 1]{CGB17}. 
In addition, by \cite[Theorem 2]{EPU13}, the first component $X$ of a CBITCL process $(X,Z)$ is a conservative stochastically continuous {\em continuous-state branching process with immigration} (CBI) in the sense of \cite{KW71}, with initial value $X_0$, immigration mechanism $\Psi$ and branching mechanism $\Phi$. Note that, since $X$ has c\`adl\`ag paths, the integral $Y:=\int_0^{\cdot}X_s\,\ud s$ is always well defined as a non-decreasing process and can therefore be used as a change of time. 

In the following, we write that a process $(X, Z)$ is $\uCBITCL(X_0, \Psi, \Lambda)$ as a shorthand notation to denote the fact that $(X, Z)$ is a CBI-time-changed L\'evy process in the sense of Definition \ref{def:cbitcl}, with $X_0\geq0$ and where the exponents $\Psi$ and $\Lambda$ are given in \eqref{eq:immigrationmechanism} and \eqref{eq:levyexp}, respectively.

\begin{rem}
Definition \ref{def:cbitcl} can be extended to processes taking values on the state space $\R^m_+\times\R^n$, for $m,n\in\N$. As shown in \cite{CGB17}, this would generate processes that belong to the affine class. However, it seems difficult to derive explicit results such as a characterization of the finiteness of exponential moments in the multi-dimensional case. For this reason and for ease of exposition, in the present work we restrict our attention to two-dimensional CBITCL processes.
\end{rem}

\subsection{CBITCL processes as solutions to stochastic integral equations}		\label{sec:integral_equation}

It is well known that a CBI process can be characterized in two equivalent ways: as the solution to the stochastic integral equation of Dawson and Li (see \cite{DL06}) and as the solution to a stochastic time change equation of Lamperti-type (see \cite{EPU13}). It can be shown that these two representations of a CBI process are equivalent in a weak sense (see \cite[Theorem 2.12]{phdthesis_szulda}).
In the present context of CBITCL processes, Definition \ref{def:cbitcl} builds upon the Lamperti-type representation of a CBI process. In this section, we show that the Dawson-Li representation can be extended to CBITCL processes. 
To this effect, let us introduce the following objects, assumed to be mutually independent:
\begin{itemize}
\item  a two-dimensional Brownian motion $B=(B_t)_{t\geq0}$ with independent components;
\item a Poisson random measure $N_0(\ud t, \ud x, \ud z)$ on $(0,+\infty)\times\R_+\times\R$ of compensator $\ud t\,\nu(\ud x)\,\delta_0(\ud z)$;
\item a Poisson random measure $N_1(\ud t, \ud u, \ud x, \ud z)$ on $(0,+\infty)^2\times\R_+\times\R$ of compensator $\ud t\,\ud u\,\Pi(\ud x, \ud z)$.
\end{itemize}
In the following, we shall always use the tilde notation to denote compensated random measures.
Let $\sigma\in\R^{2\times2}$ be a matrix with non-negative diagonal elements such that $\sigma\sigma^{\top}=\sigma^2$ and, for $X_0\geq0$, consider the following two-dimensional stochastic integral equation:
\begin{align}
\begin{pmatrix}
X_t \\
Z_t
\end{pmatrix} &= \begin{pmatrix}
X_0 \\
0
\end{pmatrix} + \int_0^t {\left(\begin{pmatrix}\beta\\0\end{pmatrix} + X_s\,b\right)\ud s}  
+ \sigma\int_0^t {\sqrt{X_s}\,\ud B_s} 
+ \int_0^t\int_{\R_+\times\R} {\begin{pmatrix}
	x \\
	0
	\end{pmatrix}N_0(\ud s, \ud x, \ud z)}
	\notag\\
&\quad+\int_0^t\int_0^{X_{s^-}}\!\int_{\R_+\times\R} {\begin{pmatrix}
	0 \\
	z\ind_{\{|z| \geq 1\}}
	\end{pmatrix}N_1(\ud s, \ud u, \ud x, \ud z)}
	\label{eq:dawsonli}\\
&\quad+\int_0^t\int_0^{X_{s^-}}\!\int_{\R_+\times\R} {\begin{pmatrix}
	x \\
	z\ind_{\{|z| < 1\}}
	\end{pmatrix}\tildeN_1(\ud s, \ud u, \ud x, \ud z)}.
	\notag
\end{align}

If $(\Omega,\cF,\FF=(\cF_t)_{t\geq0},\PP)$ is a filtered probability space supporting $B$, $N_0$ and $N_1$ with the above properties, then equation \eqref{eq:dawsonli} admits a unique strong solution $(X,Z)$ defined on $(\Omega,\cF,\FF,\PP)$. 
Indeed, by \cite[Corollary 5.2]{FL10}, there exists a unique strong solution $X$ to the first component of the stochastic integral equation \eqref{eq:dawsonli}, which corresponds to the Dawson-Li representation of the CBI process $X$, see \cite{DL06}. In turn, noting that the right-hand side of \eqref{eq:dawsonli} depends only on the process $X$, this implies the uniqueness of the process $Z$ defined by the second component of \eqref{eq:dawsonli}.

The next theorem asserts that defining CBITCL processes as solutions to the stochastic integral equation \eqref{eq:dawsonli} is equivalent to Definition \ref{def:cbitcl}.
In the following, we say that a process $(X,Z)$ defined on the probability space $(\Omega,\cF,\PP)$ is a {\em weak solution} to \eqref{eq:dawsonli} if there exists an enlargement of the probability space supporting a two-dimensional Brownian motion $B$ and Poisson random measures $N_0$ and $N_1$ as above such that $(X,Z)$ satisfies \eqref{eq:dawsonli}.
In a similar manner, we say that a process $(X,Z)$ defined on a probability space $(\Omega,\cF,\PP)$ is {\em equivalent} to a $\uCBITCL(X_0, \Psi, \Lambda)$ if there exists an enlargement of the probability space supporting independent L\'evy processes $K$ and $(M,N)$ with exponents $\Psi$ and $\Lambda$, respectively, such that $(X,Z)$ satisfies \eqref{eq:stochtimechangeequation}.

\begin{thm}\label{thm:weakequivalencecbitcl}
If the process $(X,Z)$ is a $\uCBITCL(X_0, \Psi, \Lambda)$, then it is a weak solution to \eqref{eq:dawsonli}. Conversely, if the process $(X,Z)$ is solution to \eqref{eq:dawsonli} on some filtered probability space $(\Omega,\cF,\FF,\PP)$, then it is equivalent to a $\uCBITCL(X_0, \Psi, \Lambda)$, with $\Psi$ and $\Lambda$ given by \eqref{eq:immigrationmechanism} and \eqref{eq:levyexp}, respectively.
\end{thm}
\begin{proof}
Let $(X,Z)$ be a $\uCBITCL(X_0, \Psi, \Lambda)$ and denote by $\FF^{(M,N)}$ the natural filtration of $(M,N)$. Making use of the L\'evy-It\^o decomposition of $(M,N)$, we have that
\[
\begin{pmatrix}
M_{Y_t} \\
N_{Y_t}
\end{pmatrix} = b\,Y_t  + \sigma\, W_{Y_t} + \int_0^{Y_t}\int_{\R_+\times\R} {\begin{pmatrix}
	0 \\
	z\ind_{\{|z| \geq 1\}}
	\end{pmatrix}\mu(\ud s, \ud x, \ud z)}
+\int_0^{Y_t}\int_{\R_+\times\R} {\begin{pmatrix}
	x \\
	z\ind_{\{|z| < 1\}}
	\end{pmatrix}\tilde{\mu}(\ud s, \ud x, \ud z)},
\]
for all $t\geq0$, where $(W_t)_{t\geq0}$ is two-dimensional Brownian motion in $\FF^{(M,N)}$ and $\mu(\ud t, \ud x, \ud z)$ is a Poisson random measure with compensator $\ud t\,\Pi(\ud x, \ud z)$ in $\FF^{(M,N)}$, independent of $W$. 
Let the filtration $\GG^{(M,N),K}$ be defined as the initial enlargement of $\FF^{(M,N)}$ with respect to $\sigma(K_s;s\geq0)$, i.e., $\cG^{(M,N),K}_t:=\cF^{(M,N)}_t\vee\sigma(K_s;s\geq0)$, for all $t\geq0$. 
By part (2) of \cite[Lemma 5]{EPU13}, it holds that $Y_t$ is a $\GG^{(M,N),K}$-stopping time, for each $t\geq0$. Therefore, the process $(Y_t)_{t\geq0}$ is a change of time in $\GG^{(M,N),K}$, in the sense of \cite[Definition 10.1]{jacod79}. By continuity of $Y$, the random measure $\mu(\ud t, \ud x, \ud z)$ is adapted to $Y$ (see \cite[Definition 10.25]{jacod79}) and, due to the independence of $(M,N)$ and $K$, it has the same compensator $\ud t\,\Pi(\ud x, \ud z)$ in both filtrations $\GG^{(M,N),K}$ and $\FF^{(M,N)}$. Applying \cite[Theorems 10.27 and 10.28]{jacod79}, we can perform a change of time and obtain
\begin{align*}
\int_0^{Y_t}\int_{\R_+\times\R} {\begin{pmatrix}
	0 \\
	z\ind_{\{|z| \geq 1\}}
	\end{pmatrix}\mu(\ud s, \ud x, \ud z)}
&= \int_0^t\int_{\R_+\times\R} {\begin{pmatrix}
	0 \\
	z\ind_{\{|z| \geq 1\}}
	\end{pmatrix}\mu(X_s\,\ud s, \ud x, \ud z)},\\
\int_0^{Y_t}\int_{\R_+\times\R} {\begin{pmatrix}
	x \\
	z\ind_{\{|z| < 1\}}
	\end{pmatrix}\tilde{\mu}(\ud s, \ud x, \ud z)}
&= 	\int_0^t\int_{\R_+\times\R} {\begin{pmatrix}
	x \\
	z\ind_{\{|z| < 1\}}
	\end{pmatrix}\tilde{\mu}(X_s\,\ud s, \ud x, \ud z)}, 
\end{align*}
for all $t\geq0$.
By \cite[Theorem II.7.4]{IW89}, on a suitable extension of the probability space there exists a Poisson random measure $N_1(\ud t, \ud u, \ud x, \ud z)$ with compensator $\ud t\,\ud u\,\Pi(\ud x, \ud z)$ such that
\begin{align*}
\int_0^t\int_{\R_+\times\R} {\begin{pmatrix}
	0 \\
	z\ind_{\{|z| \geq 1\}}
	\end{pmatrix}\mu(X_s\,\ud s, \ud x, \ud z)}
&= \int_0^t\int_0^{X_{s^-}}\!\int_{\R_+\times\R} {\begin{pmatrix}
	0 \\
	z\ind_{\{|z| \geq 1\}}
	\end{pmatrix}N_1(\ud s, \ud u, \ud x, \ud z)},\\
\int_0^t\int_{\R_+\times\R} {\begin{pmatrix}
	x \\
	z\ind_{\{|z| < 1\}}
	\end{pmatrix}\tilde{\mu}(X_s\,\ud s, \ud x, \ud z)}
&= \int_0^t\int_0^{X_{s^-}}\!\int_{\R_+\times\R} {\begin{pmatrix}
	x \\
	z\ind_{\{|z| < 1\}}
	\end{pmatrix}\tildeN_1(\ud s, \ud u, \ud x, \ud z)},
\end{align*}
for all $t\geq0$.
Proceeding in a similar way and making use of \cite[Theorem 7.1']{IW89}, on a suitable extension of the probability space there exists a two-dimensional Brownian motion $B=(B_t)_{t\geq0}$ such that $ W_{Y_t}=\int_0^t\sqrt{X_s}\ud B_s$, for all $t\geq0$.
We have therefore obtained that
\be	\label{eq:proof_time_change}\begin{aligned}
\begin{pmatrix}
M_{Y_t} \\
N_{Y_t}
\end{pmatrix} &= \int_0^t {X_s\,b\,\ud s}  + \sigma\int_0^t {\sqrt{X_s}\,\ud B_s} + \int_0^t\int_0^{X_{s^-}}\!\int_{\R_+\times\R} {\begin{pmatrix}
	0 \\
	z\ind_{\{|z| \geq 1\}}
	\end{pmatrix}N_1(\ud s, \ud u, \ud x, \ud z)}\\
&\quad+\int_0^t\int_0^{X_{s^-}}\!\int_{\R_+\times\R} {\begin{pmatrix}
	x \\
	z\ind_{\{|z| < 1\}}
	\end{pmatrix}\tildeN_1(\ud s, \ud u, \ud x, \ud z)}.
\end{aligned}\ee
The L\'evy-It\^o decomposition of the process $K$ can be written as $K_t = \beta t + \int_0^t\int_{\R_+\times\R} {x\,N_0(\ud s, \ud x, \ud z)}$, for all $t\geq0$, where $N_0(\ud t, \ud x, \ud z)$ is a Poisson random measure with compensator $\ud t\,\nu(\ud x)\,\delta_0(\ud z)$.
Adding the L\'evy-It\^o decomposition of $K$ to \eqref{eq:proof_time_change} and comparing with \eqref{eq:stochtimechangeequation}, we have proved that $(X,Z)$ solves equation \eqref{eq:dawsonli}, on a suitable extension of the original probability space.

Conversely, suppose that $(X,Z)$ solves \eqref{eq:dawsonli} on some filtered probability space $(\Omega,\cF,\FF,\PP)$.
We first observe that the process $K=(K_t)_{t\geq0}$ defined by $K_t:=\beta t + \int_0^t\int_{\R_+\times\R} {x\,N_0(\ud s, \ud x, \ud z)}$, for all $t\geq0$, is a L\'evy process with triplet $(\beta, 0, \nu)$ and exponent $\Psi$ given by \eqref{eq:immigrationmechanism}.
Consider then the two-dimensional process $V:=(X-K,Z)$. In order to show that $V$ is a time-changed L\'evy process, we follow the proof of \cite[Theorem 3.2]{Kallsen06}. We first note that the process $V$ is constant on each interval $[r,s]\subseteq\R_+$ such that $Y_r=Y_s$ a.s. Indeed, for any such interval, it necessarily holds that $X=0$ a.s. on $[r,s]$ (Lebesgue-a.e.), and, therefore, \eqref{eq:dawsonli}  implies that $V$ is a.s. constant on $[r,s]$.
Let $Y_{\infty} := \lim_{t\to+\infty}Y_t$, which is well-defined since the process $Y$ is non-decreasing, and define the inverse time change $\tau(t) := \inf\{ s \geq 0 : Y_s > t \}$, for all $t \geq 0$. 
Define the process $L^1=(L^1_t)_{t < Y_{\infty}}$ by\footnote{We point out that it may happen that $Y_{\infty}<+\infty$, since zero is an absorbing state for $X$ when $\Psi\equiv0$.} 
\[
L^1_t := V_{\tau(t)},
\qquad\text{ for all }t < Y_{\infty}.
\] 
By \cite[Lemma 10.14]{jacod79}, $V$ is adapted to $(\tau(t))_{t\geq0}$ and it holds that $V_t=L^1_{Y_t}$, for all $t\geq0$.
Since $(X,Z)$ solves \eqref{eq:dawsonli}, the semimartingale $V=(X-K,Z)$ has characteristics $(Yb,Y\sigma^2,\ud Y_t\,\Pi(\ud x,\ud z))$ in the filtration $\FF$ with respect to the truncation function $h$ given by $h(x):=(x_1,x_2\ind_{\{|x_2|<1\}})$, for $x=(x_1,x_2)\in\R^2$.
By \cite[Theorem 10.16]{jacod79}, $L^1$ is a semimartingale in the time-changed filtration $(\cF_{\tau(t)})_{t\geq0}$ on the stochastic interval $\llbracket 0, Y_{\infty} \llbracket$. Similarly as in the proof of \cite[Theorem 3.2]{Kallsen06}, \cite[Theorems 10.17 and 10.27]{jacod79} together with the fact that $Y_{\tau(t)}=t$, for all $t<Y_{\infty}$, imply that the semimartingale characteristics $(B,C,F)$ of $L^1$ in $(\cF_{\tau(t)})_{t\geq0}$ are given by
\be	\label{eq:char_time_change}
B_t = Y_{\tau(t)}\,b = t\,b, \quad 
C_t = Y_{\tau(t)}\,\sigma^2 = t\,\sigma^2,\quad 
F(\ud t,\ud x,\ud z) = \ud Y_{\tau(t)}\,\Pi(\ud x,\ud z) = \ud t\,\Pi(\ud x,\ud z), 
\ee
for all $t<Y_{\infty}$. 
Let  $L^2=(L_t^2)_{t \geq 0}$ be a two-dimensional L\'evy process with triplet $(b, \sigma^2, \Pi)$ and exponent \eqref{eq:levyexp}, independent of $V$ and $K$, possibly defined on a suitable enlargement of the probability space. 
Define the process $L=(L_t)_{t\geq0}$ by
\[
L_t := L^1_t\ind_{\{t<Y_{\infty}\}} + (L^1_{Y_{\infty}}+L^2_{t-Y_{\infty}})\ind_{\{t\geq Y_{\infty}\}},
\qquad\text{ for all } t\geq0,
\]
where $L^{1}_{Y_{\infty}}:=\lim_{t\to+\infty}V_t=:V_{\infty}$ on $\{Y_{\infty}<+\infty\}$, which is well-defined since $V$ is a semimartingale up to infinity on $\{Y_{\infty}<+\infty\}$, see \cite{CS05}.
Note that $Y_{\infty}$ is a stopping time in the time-changed filtration $(\cF_{\tau(t)})_{t\geq0}$ and, therefore, $L$ is a semimartingale in $(\cF_{\tau(t)})_{t\geq0}$. Moreover, by the results of \cite[Section 10.2b]{jacod79}, it can be shown that the semimartingale characteristics of $L$ are given by \eqref{eq:char_time_change}, which by \cite[Corollary II.4.19]{JS03} implies that $L$ is a L\'evy process with triplet $(b, \sigma^2, \Pi)$. 
Since $V_t=L_{Y_t}$, for all $t\geq0$, to complete the proof it remains to show that $L$ is independent of $K$.
To this effect, in view of \cite[Theorem 11.43]{HWY92}, it suffices to prove that $|\Delta L_t\,\Delta K_t|=0$ a.s. for all $t\geq0$. Using the definition of $L$ and \cite[Theorem 10.27]{jacod79}, we can compute
\be	\label{eq:stopping_times}\ba
\EE\bigl[|\Delta L_t\,\Delta K_t|\bigr]
&= \EE\bigl[\ind_{\{t<Y_{\infty}\}}|\Delta L^1_t\,\Delta K_t|\bigr]	\\
&= \EE\bigl[\ind_{\{\tau(t)<+\infty\}}|\Delta V_{\tau(t-)}\,\Delta K_t|\bigr]	\\
&\leq  \EE\bigl[\ind_{\{\tau(t-)<t\}}|\Delta V_{\tau(t-)}\,\Delta K_t|\bigr] 
+  \EE\bigl[\ind_{\{t<\tau(t-)<+\infty\}}|\Delta V_{\tau(t-)}\,\Delta K_t|\bigr], 
\ea\ee
where we have used the fact that $ \EE[\ind_{\{\tau(t-)=t\}}|\Delta V_{\tau(t-)}\,\Delta K_t|]=0$ since the random measures $N_0$ and $N_1$ in \eqref{eq:dawsonli} are assumed to be independent. Moreover, using the fact that $\tau(t-)$ is a predictable time in the filtration $\FF$ together with \cite[Proposition II.1.17]{JS03}, it holds that
\[
\EE\bigl[\ind_{\{t<\tau(t-)<+\infty\}}|\Delta V_{\tau(t-)}\,\Delta K_t|\bigr]
= \EE\bigl[\ind_{\{t<\tau(t-)<+\infty\}}|\Delta K_t|\EE[|\Delta V_{\tau(t-)}||\cF_{\tau(t-)-}]\bigr]=0.
\]
In an analogous way, since $K$ is quasi-left-continuous, we have that 
\[
\EE\bigl[\ind_{\{\tau(t-)<t\}}|\Delta V_{\tau(t-)}\,\Delta K_t|\bigr] 
= \EE\bigl[\ind_{\{\tau(t-)<t\}}|\Delta V_{\tau(t-)}|\EE[|\Delta K_t||\cF_{t-}]\bigr]=0.
\]
In view of \eqref{eq:stopping_times}, this implies that $|\Delta L_t\,\Delta K_t|=0$ a.s. for all $t\geq0$, thus completing the proof.
\end{proof}

From now on, if a CBITCL process $(X,Z)$ is directly defined as the unique strong solution to \eqref{eq:dawsonli} on some filtered probability space $(\Omega,\cF,\FF,\PP)$, we will say that $(X,Z)$ is given by its \emph{extended Dawson-Li representation}.
Representation \eqref{eq:dawsonli} can be especially useful for the numerical simulation of CBITCL processes (compare with \cite[Appendix B]{FGS21} in the case CBI processes).

\begin{rem}
One of the characteristic features of CBI processes is their self-exciting behavior. As can be deduced from equation \eqref{eq:dawsonli}, this behavior is inherited by the class of CBITCL processes. In particular, we can observe the following:
\begin{itemize}
\item The local martingale terms of the process $X$ depend on the current level of the process itself. This generates self-excitation since large values of the process increase its volatility. In particular, a large jump of $X$ increases the likelihood of further jumps of $X$.
\item The volatility of $Z$ is determined by $X$. Therefore, large values of $X$ are associated to an increased volatility of both processes, thus generating volatility clustering effects in $(X,Y)$.
\end{itemize}
As mentioned in the introduction, self-exciting and volatility clustering phenomena are often present in financial time series. Together with the remarkable analytical tractability, this makes CBITCL processes especially appropriate for financial modeling, as illustrated in \cite{FGS22} in the context of foreign currency markets with stochastic volatility (see also the examples in Section \ref{sec:examples}).
\end{rem}

The next proposition characterizes CBITCL processes in terms of their semimartingale differential characteristics (see \cite{Kallsen06} for additional information on this notion of characteristics).

\begin{prop}\label{prop:cbitclsemimartingale}	
Let $(X,Z)$ be a $\uCBITCL(X_0, \Psi, \Lambda)$. Then, $(X,Z)$ is a semimartingale with differential characteristics $(B,C,F)$ relative to the truncation function $h(x)=x\ind_{\{|x| < 1\}}$ given by
\be\begin{gathered}
B_t = \begin{pmatrix}\beta + \int_0^1 {x\,\nu(\ud x)} \\ 0\end{pmatrix} 
+ X_{t^{-}}\Biggl(b - \begin{pmatrix}
\int_{[1, +\infty)\times\R} {x\,\Pi(\ud x, \ud z)} \\
0
\end{pmatrix}\Biggr), 
\qquad
C_t = X_{t^{-}}\,\sigma^2,\\
F_t(\ud x,\ud z) = \nu(\ud x)\delta_{0}(\ud z) + X_{t^{-}}\Pi(\ud x, \ud z),
\qquad\text{ for all }t\geq0.
\label{eq:char}
\end{gathered}\ee
Conversely, if $(X,Z)$ is a semimartingale with differential characteristics given as in \eqref{eq:char}, with $\beta\geq0$, $b$ and $\sigma$ as in \eqref{eq:triplet}, $\nu$ and $\Pi$ L\'evy measures on $\R_+$ and $\R_+\times\R$, respectively, such that $\int_0^1 x\,\nu(\ud x)<+\infty$ and $\int_{[1, +\infty)\times\R} {x\,\Pi(\ud x, \ud z) < +\infty}$, then it is equivalent to a $\uCBITCL(X_0, \Psi, \Lambda)$, where $\Psi$ and $\Lambda$ are determined by $(\beta,b,\sigma^2,\nu,\Pi)$ as in \eqref{eq:immigrationmechanism} and \eqref{eq:levyexp}, respectively.
\end{prop}
\begin{proof}
Let $(X,Z)$ be a $\uCBITCL(X_0, \Psi, \Lambda)$ on $(\Omega,\cF,\PP)$. By Theorem \ref{thm:weakequivalencecbitcl}, the process $(X,Z)$ is solution to equation \eqref{eq:dawsonli} on an extended filtered probability space $(\Omega',\cF',\FF',\PP')$ (see \cite[Definition II.7.1]{IW89}). This implies that $(X,Z)$ is a quasi-left-continuous semimartingale on $(\Omega',\cF',\FF',\PP')$ and its differential characteristics as stated in \eqref{eq:char} can be directly deduced from \eqref{eq:dawsonli}. 
In view of \cite[Remark 10.40]{jacod79}, $(X,Z)$ is also a semimartingale on $(\Omega,\cF,\FF^{(X,Z)},\PP)$, where $\FF^{(X,Z)}=(\cF^{(X,Z)}_t)_{t\geq0}$ denotes the natural filtration of $(X,Z)$, with the same differential characteristics given in \eqref{eq:char}.

Conversely, suppose that $(X,Z)$ is a semimartingale with differential characteristics given as in \eqref{eq:char}. Writing the canonical representation of $(X,Z)$ (which is determined by the differential characteristics, see \cite[Theorem II.2.34]{JS03}) and arguing as in the first part of the proof of Theorem \ref{thm:weakequivalencecbitcl}, it can be easily shown that $(X,Z)$ satisfies \eqref{eq:dawsonli} on a suitable extension of the probability space. Therefore, it follows from Theorem \ref{thm:weakequivalencecbitcl} that $(X,Z)$ is equivalent to a $\uCBITCL(X_0, \Psi, \Lambda)$.
\end{proof}

\subsection{CBITCL processes as affine processes}	\label{sec:affine}

In this section, we regard CBITCL processes as affine processes in the sense of \cite{DFS03}. More specifically, we connect CBITCL processes to affine stochastic volatility processes as considered in \cite{KR11}.
We first state the following proposition, which provides the Fourier--Laplace transform of the joint process $(X,Y,Z)$, where $Y:=\int_0^{\cdot}{X_s\,\ud s}$.
When restricted to the pair $(X,Z)$, this result follows as a consequence of \cite[Theorem 1]{CGB17}.
Without loss of generality (see Theorem \ref{thm:weakequivalencecbitcl}), we suppose that the CBITCL process $(X,Z)$ is given by its extended Dawson-Li representation on a filtered probability space $(\Omega,\cF,\FF,\PP)$ supporting the processes introduced at the beginning of Section \ref{sec:integral_equation}.

\begin{prop}\label{prop:cbitclaffine}  
Let $(X, Z)$ be a $\uCBITCL(X_0, \Psi, \Lambda)$ and consider the  process $(X, Y, Z)$, where $Y:=\int_0^{\cdot} {X_s\,\ud s}$. Then, $(X,Z)$ and $(X, Y, Z)$ are affine procesess on the state spaces $\R_+\times\R$ and $\R_+^2\times\R$, respectively, and it holds that
\be\label{eq:cbitclsemigroup}
\EE\bigl[e^{u_1X_T + u_2Y_T + u_3Z_T}\big|\cF_t\bigr] 
= \exp\bigl(\cU(T-t,u_1, u_2, u_3) + \cV(T-t, u_1, u_2, u_3)X_t + u_2Y_t + u_3Z_t\bigr), 
\ee
for all $(u_1, u_2, u_3) \in \C_-^2\times\im\R$ and $0 \leq t \leq T < +\infty$, where the functions $\cU(\cdot, u_1, u_2, u_3):\R_+\rightarrow\C$ and $\cV(\cdot, u_1, u_2, u_3):\R_+\rightarrow\C_-$ are solutions to
\begin{align}
\cU(t, u_1, u_2, u_3) &= \int_0^{t} {\Psi\bigl(\cV(s, u_1, u_2, u_3)\bigr)\,\ud s},\label{eq:cbitclriccati1}\\
\frac{\partial \cV}{\partial t}(t, u_1, u_2, u_3) &= \Lambda\bigl(\cV(t, u_1, u_2, u_3),u_3\bigr) + u_2, \qquad \cV(0, u_1, u_2, u_3) = u_1,\label{eq:cbitclriccati2}
\end{align}	
where $\Psi:\C_-\rightarrow\C$ and $\Lambda:\C_-\times\im\R\rightarrow\C$ denote the functions given by \eqref{eq:immigrationmechanism} and \eqref{eq:levyexp}, respectively.
\end{prop}
\begin{proof}
Since $(X,Z)$ is the unique strong solution to \eqref{eq:dawsonli} on $(\Omega,\cF,\FF,\PP)$, it is a strong Markov process and a semimartingale with characteristics given by \eqref{eq:char} (see Proposition \ref{prop:cbitclsemimartingale}). By \cite[Theorem 2.12]{DFS03}, it follows that $(X,Z)$ is an affine process on $\R_+\times\R$ with functional characteristics $(F,R)$  given by 
\[
F(v_1, v_2) = \Psi(v_1),\qquad 
R(v_1, v_2) = \Lambda(v_1, v_2), 
\]
for all $(v_1, v_2) \in \C_-\times\im\R$.
By applying \cite[Theorem 4.10]{phdthesis_kr}, we obtain that the process $(X,Y,Z)$ is an affine process on $\R_+^2\times\R$ with functional characteristics $(\widetilde{F}, \widetilde{R})$ given by
\[
\widetilde{F}(u_1, u_2, u_3) = F(u_1, u_3) = \Psi(u_1),\qquad 
\widetilde{R}(u_1, u_2, u_3) = R(u_1, u_3) + u_2 = \Lambda(u_1,u_3) + u_2, 
\]
for all $(u_1, u_2, u_3)\in\C_-^2\times\im\R$.
The affine transform formula \eqref{eq:cbitclsemigroup} and the Riccati equations \eqref{eq:cbitclriccati1}-\eqref{eq:cbitclriccati2} are then a consequence of \cite[Theorem 2.7]{DFS03}.
\end{proof}

The availability of the explicit characterization of the conditional Fourier-Laplace transform stated in Proposition \ref{prop:cbitclaffine} is of great usefulness for the application of CBITCL processes in finance (see also Section \ref{sec:pricing}). More specifically, many pricing applications require the computation of conditional expectations of the form \eqref{eq:cbitclsemigroup}, with $Y$ typically playing the role of a discount factor. 

In line with \cite{KR11}, we say that a process $(X,Z)$ taking values in $\R_+\times\R$ is an {\em affine stochastic volatility process} if it is an affine process and its Fourier-Laplace transform has the structure \eqref{eq:cbitclsemigroup} (with $u_2=0$), for some functions $\cU$ and $\cV$. This terminology is explained by the fact that, in financial applications, the process $Z$ usually plays the role of the log-price process of a risky asset, while $X$ represents its instantaneous variance.\footnote{We point out that the notion of affine stochastic volatility process that we adopt in this paper is more general than \cite[Definition 2.8]{KR11}. Indeed, we do not require that $\exp(Z)$ is a martingale nor impose a non-degeneracy condition on $\Xi$ (corresponding to conditions (A3)-(A4) in \cite{KR11}). An explicit characterization of the martingale property of $\exp(Z)$ will be given below in Corollary \ref{cor:martingale}.} 
Proposition \ref{prop:cbitclaffine} directly implies that CBITCL processes belong to the class of affine stochastic volatility processes. In the next proposition, we provide conditions for the validity of the converse implication. 
We recall that, if $(X,Z)$ is an affine process on $\R_+\times\R$, then the compensator of its jump measure is of the form $\nu^{(X,Z)}(\ud t,\ud x,\ud z)=(m_0(\ud x,\ud z)+X_{t-}m_1(\ud x,\ud z))\ud t$, where $m_0$ and $m_1$ are L\'evy measures on $\R_+\times\R$ satisfying $\int_{\R_+\times\R}(1\wedge(x+z^2))m_0(\ud x,\ud z)<+\infty$ (see, e.g., \cite[Section 2.1]{KR11}).

\begin{prop}	\label{prop:ASVP}
Let $(X,Z)$ be an affine stochastic volatility process. Then, $(X,Z)$ is equivalent to a CBITCL process if the following three conditions hold:
\begin{enumerate}
\item[(i)] $\int_{[1,+\infty)\times\R}x\,m_1(\ud x,\ud z)<+\infty$;
\item[(ii)] $m_0(\ud x, \ud z) = m_0^X(\ud x)\,\delta_0(\ud z) + \delta_0(\ud x)\,m_0^Z(\ud z)$;
\item[(iii)] $Z$ is a.s. constant on every interval $[r,s]\subseteq\R_+$ such that $\int_r^sX_{u-}\ud u=0$ a.s.
\end{enumerate}
\end{prop}
\begin{proof}
If $(X,Z)$ is an affine stochastic volatility process, then by \cite[Theorem 2.12]{DFS03} it is a semimartingale with differential characteristics $(B,C,F)$ given by
\[
B_t=\begin{pmatrix}
\beta_1 \\ \beta_2
\end{pmatrix}
+ X_{t-}\begin{pmatrix}
b_1 \\ b_2
\end{pmatrix},
\qquad
C_t=\begin{pmatrix}
0 & 0 
\\ 0 & \alpha_2
\end{pmatrix}
+ X_{t-}\begin{pmatrix}
a_{11} & a_{12} 
\\ a_{21} & a_{22}
\end{pmatrix},
\]
\[
F_t(\ud x,\ud z) = m_0(\ud x,\ud z)+X_{t-}m_1(\ud x,\ud z),
\]
where $(\beta_1,\beta_2)\in\R_+\times\R$, $(b_1,b_2)\in\R^2$, $\alpha_2\in\R_+$ and $\bigl(\begin{smallmatrix}
a_{11} & a_{12} 
\\ a_{21} & a_{22}
\end{smallmatrix}\bigr)$ is a symmetric positive semi-definite matrix. 
By condition (iii), the process $Z$ is constant on all intervals $[r,s]\subseteq\R_+$ such that $X_{u-}=0$ a.s. for a.e. $u\in[r,s]$, which implies that $\beta_2=0$ and $\alpha_2=0$. Moreover, conditions (ii) and (iii) together imply that $m_0^Z=0$. Using condition (i) and recalling that $\int_0^1x\,m^X_0(\ud x)<+\infty$, since $(X,Z)$ is assumed to be an affine stochastic volatility process, we have thus shown that $(B,C,F)$ can be written in the form \eqref{eq:char}, where the L\'evy measures  $\nu(\ud x)$ and $\Pi(\ud x, \ud z)$ correspond to $m_0^X(\ud x)$ and $m_1(\ud x, \ud z)$, respectively. 
The claim then follows by Proposition \ref{prop:cbitclsemimartingale}.
\end{proof}

In the remaining sections of the paper, we shall restrict our attention to CBITCL processes satisfying the following assumption on the dependence structure of the L\'evy process $(M,N)$.

\begin{ass}\label{ass:no_common_jumps}
The L\'evy measure $\Pi$ is of the form 
\be\label{eq:levymeasure}
\Pi(\ud x, \ud z) = \pi(\ud x)\,\delta_0(\ud z) + \delta_0(\ud x)\,\gamma(\ud z),
\ee
where $\pi$ is a L\'evy measure on $\R_+$ such that $\int_1^{+\infty} {x\,\pi(\ud x)} < +\infty$ and $\gamma$ is a L\'evy measure on $\R$.\\
\end{ass}

Assumption \ref{ass:no_common_jumps} amounts to excluding common jumps in the processes $X$ and $Z$, while allowing for non-zero correlation in their continuous local martingale parts. This assumption enables us to derive explicit results in the following sections, in particular an explicit characterization of the lifetime of exponential moments. Such explicit results would be impossible to obtain under a general L\'evy measure, similarly to the case of general affine stochastic volatility processes (see \cite{KR11}).

Under Assumption \ref{ass:no_common_jumps}, the L\'evy exponent $\Lambda$ given in \eqref{eq:levyexp} takes the following simpler form:
\[
\Lambda(u_1, u_2) = \Phi(u_1) + \rho\,\sigma_X\sigma_Z\,u_1 u_2 + \Xi(u_2), 
\qquad \text{for all } (u_1, u_2) \in \C_-\times\im\R,
\]
where $\Phi$ and $\Xi$ are given by
\begin{align}
\Phi(u_1) &= -b_X u_1 +  \frac{1}{2}\sigma_X^2u_1^2 + \int_{\R_+}{(e^{u_1 x} - 1 - u_1 x)\pi(\ud x)}, 
\qquad \text{ for all } u_1 \in\C_-,
\label{eq:newbranching}\\
\Xi(u_2) &= b_Z u_2 + \frac{1}{2}\sigma_Z^2u_2^2 + \int_{\R}{(e^{u_2 z} - 1 - u_2 z\ind_{\{|z|<1\}})\gamma(\ud z)}, 
\qquad \text{ for all } u_2 \in \im\R.
\label{eq:newxi}
\end{align}

In the following, we shall write that a CBITCL process $(X,Z)$ is a $\uCBITCL(X_0, \Psi, \Phi, \Xi, \rho)$ if Assumption \ref{ass:no_common_jumps} holds, with $\Psi$, $\Phi$, $\Xi$ respectively given by \eqref{eq:immigrationmechanism}, \eqref{eq:newbranching}, \eqref{eq:newxi} and $\rho\in[-1,1]$.

For later use, we note that under Assumption \ref{ass:no_common_jumps} the stochastic integral equation \eqref{eq:dawsonli} takes the following form:
\begin{align}
X_t &= X_0 + \int_0^t {( \beta - b_XX_s )\,\ud s} + \sigma_X\int_0^t {\sqrt{X_s}\,\ud B_s^X} + \int_0^t \int_0^{+\infty} {x\,N_0(\ud s, \ud x)}\notag\\
&\quad + \int_0^t \int_0^{X_{s-}}\! \int_0^{+\infty} {x\,\tildeN_1(\ud s, \ud u, \ud x)},\label{eq:cbitclsde1}\\
Z_t &= b_Z\int_0^t {X_s\,\ud s} + \sigma_Z\int_0^t {\sqrt{X_s}\,\ud B_s^Z} + \int_0^t \int_0^{X_{s-}}\!\int_{|z| \geq 1} {z\,N_2(\ud s, \ud u, \ud z)}\notag\\
&\quad + \int_0^t \int_0^{X_{s-}}\!\int_{|z| < 1} {z\,\tildeN_2(\ud s, \ud u, \ud z)},\label{eq:cbitclsde2}
\end{align}
where $B^X=(B_t^X)_{t \geq 0}$ and $B^Z=(B_t^Z)_{t \geq 0}$ are one-dimensional Brownian motions with correlation $\rho$, $N_0(\ud t,\ud x)$ is a Poisson random measure on $(0,+\infty)\times\R_+$ with compensator $\ud t\,\nu(\ud x)$, $N_1(\ud t,\ud u,\ud x)$ is a Poisson random measure on $(0,+\infty)\times\R_+^2$ with compensator $\ud t\,\ud u\,\pi(\ud x)$ and $N_2(\ud t,\ud x,\ud z)$ is a Poisson random measure on $(0,+\infty)\times\R_+\times\R$ with compensator $\ud t\,\ud u\,\gamma(\ud z)$. Moreover, the random measures $N_0$, $N_1$ and $N_2$ are mutually independent and also independent of the Brownian motions $B^X$ and $B^Z$.

\section{Finiteness of exponential moments and asymptotic behavior}\label{sec:moments}

In this section, assuming the validity of Assumption \ref{ass:no_common_jumps}, we study the existence of (discounted) exponential moments of CBITCL processes and we characterize their asymptotic behavior. These properties are intimately connected to the maximal lifetime of the solutions to the Riccati equations, explicitly characterized in Theorem \ref{thm:lifetimetheoremcbitcl} below.
We let $(X,Z)$ be a $\uCBITCL(X_0, \Psi, \Phi, \Xi, \rho)$ given by its extended Dawson-Li representation on a given filtered probability space $(\Omega,\cF,\FF,\PP)$.

\subsection{Finiteness of exponential moments}	\label{sec:exponentialmomentscbitcl}

In financial applications, the finiteness of (discounted) exponential moments often represents an indispensable requirement (see, e.g., \cite{FGS21,FGS22}). 
In order to make use of some results of \cite{KRM15} for general affine processes, let us define 
\be\label{eq:domaincbi}
\cD_X := \biggl\{ u \in \R : \int_1^{+\infty} {e^{u x}\,(\nu+\pi)(\ud x)} < +\infty \biggr\}.
\ee
The convex set $\cD_X$ is non-empty and represents the effective domain of the functions $\Psi$ and $\Phi$, which can be extended to finite-valued convex functions on $\cD_X$.
Similarly, let us define
\be\label{eq:domaincbitcl}
\cD_Z := \biggl\{ u \in \R : \int_{|z| \geq 1} {e^{u z}\,\gamma(\ud z)} < +\infty \biggr\},
\ee
which represents the effective domain of the L\'evy exponent $\Xi$ when restricted to real arguments. By standard results on exponential moments of L\'evy measures (see, e.g., \cite[Theorem 25.17]{S99}), the L\'evy exponent $\Xi$ can be extended as a finite-valued convex function on $\cD_Z$.

Adapting \cite[Definition 2.10]{KRM15} to the present setup, we introduce the following definition.

\begin{defn}\label{def:extendedcbitclriccatisystem}
For every $(u_1, u_2, u_3) \in \cD_X\times\R\times\cD_Z$, we say that $(\cU(\cdot, u_1, u_2, u_3), \cV(\cdot,u_1, u_2, u_3))$ is a solution to the \emph{extended Riccati system} if it solves the following system:
\begin{align}
\cU(t, u_1, u_2, u_3) &= \int_0^{t} \Psi\bigl(\cV(s, u_1, u_2, u_3)\bigr)\ud s,\label{eq:extendedcbitclode_0}\\
\frac{\partial \cV}{\partial t}(t, u_1, u_2, u_3) &= \Phi\bigl(\cV(t,u_1,u_2, u_3)\bigr) + u_2 + \rho\,\sigma_X\,\sigma_Z\,u_3\,\cV(t,u_1,u_2, u_3) + \Xi(u_3),\label{eq:extendedcbitclode}\\
\cV(0, u_1,u_2, u_3) &= u_1,\notag
\end{align}
up to a time $\uT^{(u_1, u_2, u_3)} \in [0, +\infty]$, with $\uT^{(u_1, u_2, u_3)}$ denoting the joint lifetime of the functions $\cU(\cdot, u_1, u_2, u_3):[0, \uT^{(u_1, u_2, u_3)})\rightarrow\R$ and $\cV(\cdot, u_1, u_2, u_3):[0, \uT^{(u_1, u_2, u_3)})\rightarrow\cD_X$.  
\end{defn}

Definition \ref{def:extendedcbitclriccatisystem} extends the Riccati system \eqref{eq:cbitclriccati1}-\eqref{eq:cbitclriccati2} by allowing for the possibility of explosion in finite time. In some situations, which will be precisely characterized in Theorem \ref{thm:lifetimetheoremcbitcl} below, the lifetime $\uT^{(u_1, u_2, u_3)}$ turns out to be infinite, in which case \eqref{eq:extendedcbitclode_0}-\eqref{eq:extendedcbitclode} admit a global solution.

It is well known that the branching mechanism $\Phi$ is a locally Lipschitz continuous function on the interior $\cD_X^{\circ}$ of the set $\cD_X$, but it may fail to be so at the boundary $\partial \cD_X$. Therefore, a solution $\cV(\cdot, u_1, u_2, u_3)$ to the ODE \eqref{eq:extendedcbitclode} may not be unique when it starts at $\partial\cD_X$ or reaches it at a later time. This observation motivates the introduction of the concept of minimal solution in \cite{KRM15}. In our setup, for the sake of tractability, we prefer to impose an additional mild technical assumption which guarantees uniqueness of the (local) solution to the ODE \eqref{eq:extendedcbitclode} for every $(u_1, u_2, u_3) \in \cD_X\times\R\times\cD_Z$. To this effect, we define
\be \label{eq:d1intervalcbi}
\psi := \sup\bigl\{ x \geq 0 : \Psi(x) < +\infty \bigr\} 
\qquad \text{and} \qquad 
\phi := \sup\bigl\{ x \geq 0 : \Phi(x) < +\infty \bigr\}.
\ee
Since $\cD_X$ is a convex set containing $\R_-$, it can be written as $\cD_X=(-\infty, \psi\wedge\phi)$, or $(-\infty, \psi \wedge \phi]$ when $\Psi(\psi \wedge \phi) \vee \Phi(\psi \wedge \phi) < +\infty$ (which is equivalent to $\int_1^{+\infty} {e^{(\psi \wedge \phi)x}(\nu+\pi)(\ud x)} < +\infty$). The function $\Phi$ is a finite-valued convex function on $\cD_X$ and, hence, differentiable a.e. on $\cD_X^{\circ}$, with
\be\label{eq:derivativebranchingcbi}
\Phi^{\prime}(u) = -b_X + \sigma_X^2 u + \int_0^{+\infty} {x(e^{u x}-1)\pi(\ud x)},
\qquad \text{ for all } u \in \cD_X^{\circ}.
\ee
If $\psi \wedge \phi = +\infty$, then $\cD_X=\R$, in which case $\Phi \in \cC^1(\R,\R)$ obviously holds. If $\psi \wedge \phi < +\infty$, the following assumption ensures that $\Phi^{\prime}(\psi \wedge \phi) < +\infty$, which in turn implies that $\Phi \in \cC^1(\cD_X,\R)$.

\begin{ass}\label{ass:lipschitzcbi}
If $\psi \wedge \phi < +\infty$, then $\int_1^{+\infty} {x\,e^{(\psi \wedge \phi)x}\pi(\ud x)} < +\infty$.
\end{ass}

The validity of Assumption \ref{ass:lipschitzcbi} can be easily verified for specific types of CBITCL process. Moreover, as illustrated in Section \ref{sec:examples}, it is satisfied by a large class of models.
Under Assumption \ref{ass:lipschitzcbi}, there exists a unique solution $(\cU(\cdot, u_1, u_2, u_3), \cV(\cdot,u_1, u_2, u_3))$ to the extended Riccati system \eqref{eq:extendedcbitclode_0}-\eqref{eq:extendedcbitclode} up to time $\uT^{(u_1,u_2,u_3)}$, for all $(u_1, u_2, u_3) \in \cD_X\times\R\times\cD_Z$. This enables us to state the following result, which extends Proposition \ref{prop:cbitclaffine} and follows directly from \cite[Theorem 2.14]{KRM15}.

\begin{lem}\label{lem:extendedlaplacefouriercbitcl}
Let $(X, Z)$ be a $\uCBITCL(X_0, \Psi, \Phi, \Xi, \rho)$ and $Y:=\int_0^{\cdot} {X_s\,\ud s}$. Suppose that Assumption \ref{ass:lipschitzcbi} holds. Then, for all $(u_1, u_2, u_3) \in \cD_X\times\R\times\cD_Z$ and $0 \leq t \leq T < \uT^{(u_1, u_2, u_3)}$, it holds that
\be\label{eq:extendedcbitclsemigroup}
\EE\bigl[e^{u_1X_T + u_2Y_T + u_3Z_T}\bigr|\cF_t\bigr] = \exp\bigl( \cU(T-t, u_1, u_2, u_3) + \cV(T-t, u_1, u_2, u_3)X_t + u_2Y_t + u_3Z_t \bigr),  
\ee
where $(\cU(\cdot, u_1, u_2, u_3), \cV(\cdot,u_1, u_2, u_3))$ is the unique solution to the extended Riccati system \eqref{eq:extendedcbitclode_0}-\eqref{eq:extendedcbitclode} defined up to time $\uT^{(u_1, u_2, u_3)}$.
\end{lem}

The lifetime $\uT^{(u_1, u_2, u_3)}$ is closely connected to the finiteness of exponential moments of the process $(X,Y,Z)$. Indeed, in view of \cite[Proposition 3.3]{KRM15}, it holds that
\be\label{eq:characterizationmaximumjointlifetimecbitcl}
\uT^{(u_1, u_2, u_3)} = \sup\bigl\{ t \geq 0 : \EE[e^{u_1X_t + u_2Y_t + u_3Z_t}] < +\infty \bigr\}.
\ee

The next theorem is the main result of this section and provides an explicit formula for $\uT^{(u_1, u_2, u_3)}$, for every $(u_1, u_2, u_3)\in\cD_X\times\R\times\cD_Z$. The proof is based on techniques similar to \cite[Theorem 4.1]{KR11}, which covers the case of affine stochastic volatility processes. However, our theorem is specific to CBITCL processes and avoids the additional assumptions of \cite[Theorem 4.1]{KR11}. In particular, it allows for CBI processes $X$ with an arbitrary (not necessarily strictly subcritical, i.e. $b_X>0$) branching mechanism $\Phi$.
For $(u_2, u_3) \in \R\times\cD_Z$, we introduce the following notation:
\begin{align*}
\cS^{(u_2,u_3)} &:= \bigl\{ x \in \cD_X : \Phi(x) + u_2 + \rho\,\sigma_X\,\sigma_Z\,u_3\,x + \Xi(u_3) \leq 0 \bigr\},\\
\chi^{(u_2,u_3)} &:= \sup\cS^{(u_2,u_3)} \in [-\infty, \psi \wedge \phi],
\end{align*}
with the convention $\chi^{(u_2,u_3)} = -\infty$ if the set $\cS^{(u_2,u_3)}$ is empty. 

\begin{thm}\label{thm:lifetimetheoremcbitcl}
Let $(X, Z)$ be a $\uCBITCL(X_0, \Psi, \Phi, \Xi, \rho)$ and suppose that Assumption \ref{ass:lipschitzcbi} holds. Then, for every $(u_1, u_2, u_3) \in \cD_X\times\R\times\cD_Z$, the lifetime $\uT^{(u_1, u_2, u_3)}$ is given as follows:
\begin{itemize}
\item[(i)] if $u_1 \leq \chi^{(u_2,u_3)}$, then $\uT^{(u_1, u_2, u_3)} = +\infty$;
\item[(ii)] if $u_1 > \chi^{(u_2,u_3)}$, then 
\be\label{eq:explosionlifetimecbitcl}
\uT^{(u_1, u_2, u_3)} = \int_{u_1}^{\psi \wedge \phi} {\frac{\ud x}{\Phi(x) + u_2 + \rho\,\sigma_X\,\sigma_Z\,u_3\,x + \Xi(u_3)}}.
\ee
\end{itemize}
\end{thm}
\begin{proof}
For simplicity of notation, for fixed $(u_1,u_2,u_3)\in\cD_X\times\R\times\cD_Z$, we denote by $\uT^{\cU}$ and $\uT^{\cV}$ the lifetimes of the functions $\cU(\cdot,u_1,u_2,u_3)$ and $\cV(\cdot,u_1,u_2,u_3)$ solutions to \eqref{eq:extendedcbitclode_0}-\eqref{eq:extendedcbitclode}, respectively.
Making use of this notation, the lifetime $\uT^{(u_1,u_2,u_3)}$ can be decomposed as $\uT^{(u_1,u_2,u_3)}=\uT^{\cU}\wedge\uT^{\cV}$.
Always for simplicity of notation, we omit to write the superscript $(u_2,u_3)$ in $\chi^{(u_2,u_3)}$ and $\cS^{(u_2,u_3)}$.
Let us first consider the case $u_1\leq\chi$. If $\Phi(u_1)+u_2+\rho\,\sigma_X\,\sigma_Z\,u_1\,u_3+\Xi(u_3)=0$, then the constant function $\cV(\cdot,u_1,u_2,u_3)\equiv u_1$ is the unique solution to \eqref{eq:extendedcbitclode}, so that $\uT^{\cV}=+\infty$. Since $u_1\in\cD_X$, we also have $\uT^{\cU}=+\infty$, which implies that $\uT^{(u_1,u_2,u_3)}=+\infty$.
Suppose now that $\Phi(u_1)+u_2+\rho\,\sigma_X\,\sigma_Z\,u_1\,u_3+\Xi(u_3)<0$. Let us define $\xi:=\inf\cS$, with $\xi=+\infty$ if the set $\cS$ is empty. Note that, in the present case, $\xi<u_1$ and, if $\xi>-\infty$, then $\Phi(\xi)+u_2+\rho\,\sigma_X\,\sigma_Z\,u_3\,\xi+\Xi(u_3)=0$.
By convexity of $\Phi$, it holds that $\Phi(x)+u_2+\rho\,\sigma_X\,\sigma_Z\,u_3\,x+\Xi(u_3)<0$, for all $x\in(\xi,u_1]$. Therefore, equation (3.3) implies that the function $\cV(\cdot,u_1,u_2,u_3)$ is strictly decreasing and we can write
\[
t = -\int_{\cV(t,u_1,u_2,u_3)}^{u_1}\frac{\ud x}{\Phi(x) + u_2 + \rho\,\sigma_X\,\sigma_Z\,u_3\,x + \Xi(u_3)},
\qquad\text{ for all }t\geq0.
\]
Letting $t\to+\infty$ on both sides of this identity, we obtain that $\cV(t,u_1,u_2,u_3)\to\xi$ as $t\to+\infty$, while $\xi<\cV(t,u_1,u_2,u_3)\leq u_1$ for all $t\geq0$. This shows that $\uT^{\cV}=+\infty$. Moreover, making use of the structure of $\Psi$, we obtain $-\infty<\cU(t,u_1,u_2,u_3)\leq t\Psi(u_1)$ for all $t\geq0$, implying that $\uT^{\cU}=+\infty$. We have thus shown that $\uT^{(u_1,u_2,u_3)}=+\infty$.
If $u_1\leq\chi$ and $\Phi(u_1)+u_2+\rho\,\sigma_X\,\sigma_Z\,u_1\,u_3+\Xi(u_3)>0$, then we necessarily have $u_1<\xi\in\cD_X$ and $\Phi(x)+u_2+\rho\,\sigma_X\,\sigma_Z\,u_3\,x+\Xi(u_3)>0$, for all $x\in[u_1,\xi)$. Arguing similarly as above, this implies that $u_1\leq \cV(t,u_1,u_2,u_3)<\xi$ for all $t\geq0$, which in turn leads to $\uT^{(u_1,u_2,u_3)}=+\infty$.

Let us now consider the case $u_1>\chi$ (which includes the case $\chi=-\infty$).  
By using the convexity of $\Phi$, we have $\Phi(u_1)+u_2+\rho\,\sigma_X\,\sigma_Z\,u_1\,u_3+\Xi(u_3)> 0$, implying that the function $\cV(\cdot, u_1,u_2,u_3)$ is strictly increasing with values in $[u_1,\phi]$. The function $\cV(\cdot, u_1,u_2,u_3)$ can be extended to a maximal interval of existence $[0, \uT^*)$ such that one of the following two cases occurs:
\begin{enumerate}
\item[(i)] $\uT^*=+\infty$;
\item[(ii)] $\uT^*<+\infty$ and $\lim_{t\to\uT^*}\cV(t,u_1,u_2,u_3)=\phi$.
\end{enumerate}
In case (i), since $\cV(\cdot, u_1,u_2,u_3)$ is strictly increasing, the limit $l:=\lim_{t \to +\infty}\cV(t,u_1,u_2,u_3)$ is well-defined with values in $(u_1,\phi]\cup\{+\infty\}$. Suppose that $l < +\infty$, i.e., the line $y=l$ is a horizontal asymptote for $\cV(\cdot,u_1,u_2,u_3)$ as $t\to+\infty$. This implies that $\frac{\partial \cV}{\partial t}(t,u_1,u_2,u_3)\to0$ as $t\to+\infty$. Letting $t\to+\infty$ on both sides of  (3.3), this yields $\Phi(l) + u_2 + \rho\,\sigma_X\,\sigma_Z\,u_3\,l + \Xi(u_3) = 0$, contradicting the fact that $\Phi(x)+u_2+\rho\,\sigma_X\,\sigma_Z\,u_3\,x+\Xi(u_3) > 0$ for all $x > \chi$. Therefore, the limit $l$ must necessarily be infinite, which can only happen if $\phi=+\infty$ and, in this case, $\lim_{t\to+\infty}\cV(t,u_1,u_2,u_3)=\phi$, analogously to case (ii). In case (ii), let $(\uT_n)_{n\in\N}$ be an increasing sequence such that $\uT_n\to\uT^*$ as $n\to+\infty$.
Similarly as above, making use of equation (3.3), we can write
\be	\label{eq:n}
\uT_n = \int_{u_1}^{\cV(\uT_n,u_1,u_2,u_3)} {\frac{\ud x}{\Phi(x)+u_2+\rho\,\sigma_X\,\sigma_Z\,u_3\,x+\Xi(u_3)}},
\qquad\text{ for all }n\in\N.
\ee
Letting $n\to+\infty$ on both sides of \eqref{eq:n} yields
\[
\uT^* = \int_{u_1}^{\phi} {\frac{\ud x}{\Phi(x)+u_2+\rho\,\sigma_X\,\sigma_Z\,u_3\,x+\Xi(u_3)}},
\]
which represents the lifetime $\uT^{\cV}$ of the function $\cV(\cdot,u_1,u_2,u_3)$. 
To complete the proof, it suffices to observe that, if $\phi\leq\psi$, then $\int_0^t\Psi(\cV(s,u_1,u_2,u_3))\ud s$ is always finite whenever $\cV(t,u_1,u_2,u_3)$ is finite, so that $\uT^{(u_1,u_2,u_3)}=\uT^{\cV}$.
If $\phi>\psi$, then $\uT^{(u_1,u_2,u_3)}=\inf\{t\in\R_+:\cV(t,u_1,u_2,u_3)=\psi\}$. Without loss of generality, we can assume that there exists $n\in\N$ such that $\uT_n=\uT^{(u_1,u_2,u_3)}$ and $\cV(\uT_n,u_1,u_2,u_3)=\psi$. Inserting this into equation \eqref{eq:n} and combining the two cases $\phi\leq\psi$ and $\psi<\phi$ we obtain  formula \eqref{eq:explosionlifetimecbitcl}.
\end{proof}

\begin{rem}	\label{rem:expmoments}
(1)  The result of Theorem \ref{thm:lifetimetheoremcbitcl} is useful for financial applications (see Section \ref{sec:pricing}). Indeed, many derivatives can be efficiently priced by resorting to Fourier representations of their payoffs and exploiting the knowledge of the conditional characteristic function 
of $(X,Y,Z)$ (see, e.g., \cite[Section 10.3]{Fil09}). This requires an extension of \eqref{eq:extendedcbitclsemigroup} to the complex domain. As shown in \cite{KRM15}, the feasibility of this extension crucially depends on the fact that the lifetime $\uT^{(u_1,u_2,u_3)}$ is greater than the maturity of the payoff to be priced, for suitable $(u_1,u_2,u_3)\in\cD_X\times\R\times\cD_Z$. 

(2) In asset pricing models, the finiteness of the lifetime $\uT^{(u_1,u_2,u_3)}$, for $(u_1,u_2,u_3)\in\cD_X\times\R\times\cD_Z$,  is intimately related to the shape of the implied volatility smile at extreme strikes, see \cite{Lee04} and \cite[Section 5.1]{KR11}. Therefore, the availability of an explicit description of $\uT^{(u_1,u_2,u_3)}$ in Theorem \ref{thm:lifetimetheoremcbitcl} permits to characterize the tail behavior of the implied volatility smile in financial models driven by CBITCL processes, as will be illustrated in the examples considered in Section \ref{sec:examples}.

(3) In the case of classical CBI processes, a characterization of the lifetime of exponential moments has been obtained in \cite[Theorem 2.7]{FGS21}, which can be recovered as a special case of Theorem \ref{thm:lifetimetheoremcbitcl} by taking $u_2\in\R_-$ and $u_3=0$.
\end{rem}

The following corollary provides a necessary and sufficient condition for the finiteness of exponential moments for all $u_1\in\cD_X$, for fixed but arbitrary $(u_2,u_3)\in\R\times\cD_Z$. 
Whenever $\psi\wedge\phi=+\infty$, we denote $\Phi(\psi\wedge\phi):=\lim_{u\to+\infty}\Phi(u)$, which is well-defined with values in $\{-\infty,+\infty\}$.

\begin{cor}\label{cor:lifetimetheoremcbitcl}
Let $(X, Z)$ be a $\uCBITCL(X_0, \Psi, \Phi, \Xi, \rho)$ and suppose that Assumption \ref{ass:lipschitzcbi} holds. 
Assume that, if $\psi<+\infty$ and $\psi\leq\phi$, then $\int_1^{+\infty}e^{\psi x}\nu(\ud x)<+\infty$.
Let $(u_2, u_3) \in \R\times\cD_Z$. Then, $\uT^{(u_1, u_2, u_3)} = +\infty$ holds for all $u_1 \in \cD_X$ if and only if $\Phi(\psi \wedge \phi) + u_2 + \rho\,\sigma_X\,\sigma_Z\,u_3(\psi \wedge \phi) + \Xi(u_3) \leq 0$.
\end{cor}
\begin{proof}
Note first that, under the present assumptions, $\cD_X=(-\infty,\psi\wedge\phi]$ whenever $\psi\wedge\phi<+\infty$.
Suppose first that $\Phi(\psi \wedge \phi) + u_2 + \rho\,\sigma_X\,\sigma_Z\,u_3(\psi \wedge \phi) + \Xi(u_3) \leq 0$. In this case, $\chi^{(u_2,u_3)}=\psi\wedge\phi$ (in both cases $\psi\wedge\phi<+\infty$ and $\psi\wedge\phi=+\infty$). By Theorem \ref{thm:lifetimetheoremcbitcl}, it follows that $\uT^{(u_1,u_2,u_3)}=+\infty$ for all $u_1\in\cD_X$.
Conversely, suppose that $\uT^{(u_1,u_2,u_3)}=+\infty$ for all $u_1\in\cD_X$. If $\psi\wedge\phi<+\infty$, then $\psi\wedge\phi\in\cD_X$ and, therefore, $\uT^{(\psi\wedge\phi,u_2,u_3)}=+\infty$. Arguing by contradiction, suppose that $\Phi(\psi \wedge \phi) + u_2 + \rho\,\sigma_X\,\sigma_Z\,u_3(\psi \wedge \phi) + \Xi(u_3) > 0$. In that case, by the properties of the function $\Phi$, we would have $\chi^{(u_2,u_3)}<\psi\wedge\phi$. But then formula \eqref{eq:explosionlifetimecbitcl} would imply that $\uT^{(\psi\wedge\phi,u_2,u_3)}=0$, thus leading to a contradiction.
On the other hand, if $\psi\wedge\phi=+\infty$, then $\uT^{(u_1,u_2,u_3)}=+\infty$ for all $u_1\in\cD_X=\R$. Arguing again by contradiction, suppose that $\Phi(\psi\wedge\phi) + u_2 + \rho\,\sigma_X\,\sigma_Z\,u_3(\psi \wedge \phi) + \Xi(u_3)>0$. In this case, there exists $M>0$ such that $\Phi(x) + u_2 + \rho\,\sigma_X\,\sigma_Z\,u_3\,x + \Xi(u_3)>0$ for all $x\geq M$. In turn, this yields $\chi^{(u_2,u_3)}<M$, which by Theorem \ref{thm:lifetimetheoremcbitcl} would imply that $\uT^{(M,u_2,u_3)}<+\infty$, thus leading to a contradiction. 
\end{proof}

\subsection{Asymptotic behavior of CBITCL processes}	\label{sec:asymptotic}

In this section, we study the long-term behavior of a CBITCL process $(X,Z)$ satisfying Assumption \ref{ass:no_common_jumps}. 
Let us first recall that, if the CBI process $X$ is strictly subcritical (i.e., $b_X>0$ in the branching mechanism \eqref{eq:newbranching}), then it converges in law to a unique stationary distribution $\eta$ as $t\to+\infty$, with Laplace transform
\[
L_{\eta}(\lambda) = \exp\left(\int_{\lambda}^0\frac{\Psi(x)}{\Phi(x)}\ud x\right),
\qquad\text{ for all }\lambda\leq0,
\]
see \cite[Theorem 10.4]{Li20}. 
In general, the time-changed L\'evy process $Z$ does not admit an ergodic distribution. However, similarly as in \cite[Section 3.2]{KR11} but under weaker technical requirements, we can prove that the rescaled cumulant generating function $\frac{1}{t}\log\EE[e^{u Z_t}]$ converges as $t\to+\infty$ to a limit corresponding to the cumulant generating function of an infinitely divisible random variable.

Let $(X,Z)$ be a $\uCBITCL(X_0, \Psi, \Phi, \Xi, \rho)$ and suppose that Assumption \ref{ass:lipschitzcbi} is satisfied. We recall from Lemma \ref{lem:extendedlaplacefouriercbitcl} that 
\[
\EE\bigl[e^{uZ_t}\bigr] = \exp\bigl(\cU(t, 0, 0, u) + \cV(t, 0, 0, u)X_0\bigr), 
\qquad\text{ for all }u \in \cD_Z,
\]
where $(\cU(\cdot, 0, 0, u), \cV(\cdot, 0, 0, u))$ is the unique solution to the following extended Riccati system:
\begin{align}
\cU(t, 0, 0, u) &= \int_0^{t} {\Psi\bigl(\cV(s, 0, 0, u)\bigr)\ud s},\label{eq:functionU(0,u)}\\
\frac{\partial \cV}{\partial t}(t, 0, 0, u) &= \Phi\bigl(\cV(t, 0, 0, u)\bigr) 
+ \rho\,\sigma_X\,\sigma_Z\,u\,\cV(t, 0, 0, u)
+ \Xi(u), \qquad \cV(0, 0, 0, u) = 0,\label{eq:functionV(0,u)}
\end{align}	
for all $0 \leq t < \uT(u)$, where $\uT(u):=\uT^{(0, 0, u)}$ denotes the  maximal lifetime of $(\cU(\cdot, 0, 0, u),\cV(\cdot, 0, 0, u))$. 
The next corollary follows directly from Theorem \ref{thm:lifetimetheoremcbitcl} and provides an explicit description of $\uT(u)$. For simplicity of notation, we denote $\chi(u):=\chi^{(0,u)}$.

\begin{cor}\label{cor:lifetimeTu}
Let $(X,Z)$ be a $\uCBITCL(X_0, \Psi, \Phi, \Xi, \rho)$ and suppose that Assumption \ref{ass:lipschitzcbi} holds. Then, for every $u \in \cD_Z$, the lifetime $\uT(u)$ is given as follows:
\begin{enumerate}
\item[(i)] if $\chi(u) \geq 0$, then $\uT(u) = +\infty$;
\item[(ii)] if $\chi(u) < 0$, then 
\be\label{eq:lifetimeTufinite}
\uT(u) = \int_0^{\psi\wedge\phi} {\frac{\ud x}{\Phi(x) + \rho\,\sigma_Z\,\sigma_Z\,u\,x + \Xi(u)}}.
\ee
\end{enumerate}
\end{cor}

By Corollary \ref{cor:lifetimeTu}, $\uT(u) = +\infty$ for all $u\in\cD_Z$ such that $\chi(u)\geq0$, meaning that the functions $\cU(t,0,0,u)$ and $\cV(t,0,0,u)$ are finite for all $t\geq0$. The study of the asymptotic behavior of $\EE[e^{uZ_t}]$ therefore requires analysing the asymptotic properties of the functions $\cU(\cdot,0,0,u)$ and $\cV(\cdot,0,0,u)$ for all $u\in\cX:= \{u \in \cD_Z : \chi(u)\geq0\}$. This is the content of the next proposition, which specializes \cite[Theorem 3.4]{KR11} to the case of CBITCL processes. More precisely, by relying on Corollary \ref{cor:lifetimeTu} and exploiting the specific structure of a CBITCL process, we obtain an asymptotic result which only requires the CBI process $X$ to be strictly subcritical, besides the technical requirement of Assumption \ref{ass:lipschitzcbi}, thereby weakening some of the assumptions of \cite{KR11}.

\begin{prop}\label{prop:longtermbehavior}
Let $(X,Z)$ be a $\uCBITCL(X_0, \Psi, \Phi, \Xi, \rho)$ with $b_X>0$ and suppose that Assumption \ref{ass:lipschitzcbi} holds. For every $u\in\cX$, define $\xi(u) := \inf\bigl\{x\in \cD_X : \Phi(x) + \rho\,\sigma_Z\,\sigma_Z\,u\,x + \Xi(u) \leq 0 \bigr\}$ if $\Xi(u)\neq0$ and $\xi(u):=0$ otherwise. Then, for every $u\in\cX$, it holds that
\[
\underset{t\to+\infty}{\lim}\,\cV(t, 0, 0, u) = \xi(u)
\qquad\text{and}\qquad
\underset{t\to+\infty}{\lim}\,\frac{1}{t}\,\cU(t, 0, 0, u) = \Psi\bigl(\xi(u)\bigr).
\]
\end{prop}
\begin{proof}
The branching mechanism $\Phi$ satisfies $\Phi(0)=0$ and is continuous and convex. Moreover, if $b_X>0$, then $\lim_{x\to-\infty} (\Phi(x) + \rho\,\sigma_X\,\sigma_Z\,u\,x) = +\infty$, for every $u\in\cX$. 
Making use of these properties, the fact that $\{x\in \cD_X : \Phi(x) + \rho\,\sigma_X\,\sigma_Z\,u\,x + \Xi(u) \leq 0 \}\neq\emptyset$, for every $u\in\cX$, implies that the quantity $\xi(u)$ is always finite and belongs to $\cD_X$. In addition, by continuity of $\Phi$, it holds that
\be\label{eq:equationXi}
\Phi\bigl(\xi(u)\bigr) + \rho\,\sigma_X\,\sigma_Z\,u\,\xi(u) + \Xi(u) = 0, 
\qquad \text{for all } u\in\cX.
\ee 
Let us now consider separately the three cases $\xi(u)=0$, $\xi(u)<0$ and $\xi(u)>0$. If $\xi(u)=0$, then $\Xi(u)=0$ by \eqref{eq:equationXi} or by definition of $\xi(u)$ and the function $\cV(\cdot,0,0,u)\equiv0$ is the unique solution to \eqref{eq:functionV(0,u)}, so that $\cV(t, 0, 0, u) \to \xi(u)$ as $t \to +\infty$ trivially holds. 
If $\xi(u)<0$, then the fact that $\chi(u)\geq0$ together with the convexity of $\Phi$ implies that $\Xi(u)<0$ necessarily holds. Therefore, by convexity of $\Phi$, it holds that $\Phi(x) + \rho\,\sigma_X\,\sigma_Z\,u\,x + \Xi(u) < 0$ for all $x\in(\xi(u),0]$. Equation \eqref{eq:functionV(0,u)} then implies that $\cV(\cdot, 0, 0, u)$ is strictly decreasing and satisfies $\xi(u) < \cV(t,0,0,u) \leq 0$ and
\[
t = \int_{\cV(t,0,0,u)}^0 {\frac{-\ud x}{\Phi(x) + \rho\,\sigma_X\,\sigma_Z\,u\,x + \Xi(u)}},
\qquad\text{ for all }t\geq0.
\]
Letting $t\to+\infty$ on both sides of this identity and recalling \eqref{eq:equationXi}, we obtain $\cV(t, 0, 0, u) \to \xi(u)$ as $t \to +\infty$.
If $\xi(u)>0$, then we necessarily have $\Phi(x) + \rho\,\sigma_X\,\sigma_Z\,u\,x + \Xi(u) > 0$, for all $x\in[0,\xi(u))$, implying that $\cV(\cdot, 0, 0, u)$ is strictly increasing and satisfies $0\leq\cV(t,0,0,u)<\xi(u)$ and
\[
t = \int_{\cV(t,0,0,u)}^0 {\frac{-\ud x}{\Phi(x) + \rho\,\sigma_X\,\sigma_Z\,u\,x + \Xi(u)}},
\qquad\text{ for all }t\geq0.
\]
Analogously to the preceding case, letting $t\to+\infty$ on both sides of the latter identity and making use of \eqref{eq:equationXi}, we obtain that $\cV(t, 0, 0,u) \to \xi(u)$ as $t \to +\infty$.
Finally, for all $u\in\cX$, the convergence of the function $t\mapsto(1/t)\,\cU(t, 0, 0,u)$ directly follows from equation \eqref{eq:functionU(0,u)}:
\[
\frac{1}{t}\,\cU(t, 0, 0,u) = \frac{1}{t}\,\int_0^t {\Psi\bigl(\cV(s, 0,0,u)\bigr)\,\ud s} \longrightarrow \Psi\bigl(\xi(u)\bigr)
\qquad\text{ as }t\to+\infty.
\]
\end{proof}

Proposition \ref{prop:longtermbehavior} yields the following long-term behavior of the time-changed L\'evy process $Z$:
\[
\frac{1}{t}\log\EE\bigl[e^{uZ_t}\bigr] 
= \frac{1}{t}\,\cU(t, 0, 0,u) + \frac{1}{t}\,\cV(t, 0, 0,u)X_0 
\underset{t\to+\infty}{\longrightarrow} \Psi\bigl(\xi(u)\bigr), 
\qquad\text{ for all $u \in \cX$.}
\]
Similarly as in \cite[Theorem 3.4]{KR11}, it can be shown that $\xi(\cdot)$ and $\Psi(\xi(\cdot))$ are cumulant generating functions of infinitely divisible random variables. 
We can therefore conclude that the marginal distributions of $Z$ are asymptotically equivalent to those of a L\'evy process with characteristic exponent $\Psi(\xi(\cdot))$. 
Notice that $\Psi(\xi(\cdot))$ corresponds to the exponent obtained by subordinating a L\'evy process with exponent $\xi$ by an independent L\'evy process with exponent $\Psi$, see \cite[Theorem 30.1]{S99}. In particular, this subordinator is equivalent to the L\'evy process $K=(K_t)_{t \geq 0}$ appearing in the Lamperti-type representation of the CBI process $X$ (recall the stochastic time change equation \eqref{eq:stochtimechangeequation}, see also \cite{EPU13} and \cite[Section 2.5]{phdthesis_szulda}).

\section{CBITCL-preserving changes of probability}
\label{sec:girsanovcbitcl}

In this section, we describe a class of equivalent changes of probability that leave invariant the class of CBITCL processes. More precisely, we consider Esscher-type changes of measure under which a CBITCL process remains a CBITCL process, with modified branching and immigration mechanisms and L\'evy exponent.
The results of this section are motivated by financial applications, where one typically wants to ensure that a model preserves its structural characteristics under both the statistical and the risk-neutral probability, as well as under risk-neutral probabilities associated to different num\'eraires (see \cite{FGS22} for an application to a multi-currency market).

Similarly as in Section \ref{sec:moments}, we suppose that Assumption \ref{ass:no_common_jumps} holds. Let us fix two constants $\zeta \in \R$ and $\lambda \in \R$ and consider the process $\cW=(\cW_t)_{t \geq 0}$ defined by
\be\label{eq:processW}
\cW_t := \zeta(X_t-X_0) + \lambda Z_t, 
\qquad \text{ for all } t \geq 0.
\ee 
In view of \cite[Proposition II.8.26]{JS03}, the process $\cW$ is an exponentially special semimartingale if and only if $\zeta \in \cD_X$ and $\lambda \in \cD_Z$. In such a case, $\cW$ admits a unique exponential compensator, i.e., a predictable finite variation process $\cK=(\cK_t)_{t \geq 0}$ such that $\exp(\cW-\cK)$ is a local martingale. 
The following lemma provides the explicit representation of the exponential compensator $\cK$.

\begin{lem}\label{lem:exponentialcompensator}
Let $(X, Z)$ be a $\uCBITCL(X_0, \Psi, \Phi, \Xi, \rho)$. Consider the process $\cW$ defined by \eqref{eq:processW}, with $\zeta \in \cD_X$ and $\lambda \in \cD_Z$. Then, the exponential compensator $\cK$ of $\cW$ is given by 
\be\label{eq:exponentialcompensator}
\cK_t = t\,\Psi(\zeta) + Y_t\,\bigl(\Phi(\zeta) + \zeta\,\lambda\,\rho\,\sigma_X\,\sigma_Z + \Xi(\lambda)\bigr), 
\qquad \text{ for all } t \geq 0.
\ee
\end{lem}
\begin{proof}
In view of \cite[Theorems 2.18 and 2.19]{KS02cumulant}, taking into account that CBITCL processes are quasi-left-continuous, the exponential compensator $\cK$ coincides with the modified Laplace cumulant process of $(X,Z)$ computed at $\theta=(\zeta,\lambda)$. The latter can be explicitly expressed as follows in terms of the semimartingale differential characteristics $(B,C,F)$ of $(X,Z)$:
\[
\cK_t = \int_0^t \biggl(\theta^{\top}B_s + \frac{1}{2}\theta^{\top}C_s\,\theta + \int_{\R^2} {\bigl(e^{\theta^{\top}x} - 1 - \theta^{\top}x\ind_{\{|x| < 1\}}\bigr)\,F_s(\ud x)}\biggr)\ud s.
\]
Making use of the explicit representation \eqref{eq:char} of the characteristics $(B,C,F)$, we then obtain
\begin{align*}
\cK_t &= t\biggl(\beta\zeta + \int_0^{+\infty} {(e^{\zeta x} - 1)\nu(\ud x)}\biggr) + Y_t\biggl(-b_X\zeta + \frac{1}{2}\sigma^2_X\zeta^2 + \int_0^{+\infty} {(e^{\zeta x} - 1 - \zeta x)\pi(\ud x)}\biggr) \\
&\quad+ Y_t\,\zeta\,\lambda\,\rho\,\sigma_X\,\sigma_Z + Y_t\biggl(b_Z\lambda + \frac{1}{2}\sigma_Z^2\lambda^2 + \int_{\R} {(e^{\lambda z} - 1 - \lambda z\ind_{\{|z|<1\}})\gamma(\ud z)}\biggr) \\
&= t\,\Psi(\zeta) + Y_t\,\bigl(\Phi(\zeta) + \zeta\,\lambda\,\rho\,\sigma_X\,\sigma_Z + \Xi(\lambda)\bigr).
\end{align*}
\end{proof}

The process $\cW$ introduced in \eqref{eq:processW} can be used to define an equivalent change of probability that leaves invariant the class of CBITCL processes. To this effect, we consider a time horizon $\horizon < +\infty$ and assume that $(X, Z)$ is directly given by its extended Dawson-Li representation \eqref{eq:dawsonli} on a filtered probability space $(\Omega,\cF,\F,\PP)$. 
In view of Theorem \ref{thm:weakequivalencecbitcl}, this entails no loss of generality.

\begin{thm}\label{thm:girsanovcbitcl}
Let $(X, Z)$ be a $\uCBITCL(X_0, \Psi, \Phi, \Xi, \rho)$ and suppose that Assumption \ref{ass:lipschitzcbi} holds. Consider the process $\cW$ defined in \eqref{eq:processW}, with $\zeta \in \cD_X$ and $\lambda \in \cD_Z$, and its exponential compensator $\cK$ given by \eqref{eq:exponentialcompensator}. Then, the process $(\exp(\cW_t-\cK_t))_{t\in[0,\horizon]}$ is a martingale.
Moreover, under the probability measure $\PP^{\prime}\sim\PP$ defined on $(\Omega,\cF)$ by
\be\label{eq:measurechangecbitcl}
\frac{\ud\PP^{\prime}}{\ud\PP} := e^{\cW_{\horizon} - \cK_{\horizon}}, 
\ee 
the process $(X, Z)$ is equivalent to a CBITCL process, with parameters $\beta^{\prime}$, $\nu^{\prime}$, $b_X^{\prime}$, $\sigma_X^{\prime}$, $\pi^{\prime}$, $\rho^{\prime}$, $b_Z^{\prime}$, $\sigma_Z^{\prime}$, and $\gamma^{\prime}$ reported in Table \ref{table:newparameterscbitcl}, and Assumption \ref{ass:no_common_jumps} remains satisfied under $\PP^{\prime}$.
\end{thm}

\begin{table}[b] 
	\centering
	\captionsetup{justification=centering}
	\begin{tabular}{|l|}
		\hline
		CBITCL parameters under $\PP^{\prime}$ \\
		\hline
		\hline
		$\beta^{\prime} := \beta$ \\
		\hline
		$\nu^{\prime}(\ud x) := e^{\zeta x}\nu(\ud x)$ \\
		\hline
		$b_X^{\prime} := b_X - \zeta\,\sigma_X^2 - \lambda\,\rho\,\sigma_X\sigma_Z - \int_0^{+\infty} {x(e^{\zeta x}-1)\pi(\ud x)}$ \\ 
		\hline
		$\sigma_X^{\prime} := \sigma_X$ \\
		\hline
		$\pi^{\prime}(\ud x) := e^{\zeta x}\pi(\ud x)$ \\
		\hline
		$\rho^{\prime} := \rho$ \\
		\hline
		$b_Z^{\prime} := b_Z + \zeta\,\rho\,\sigma_X\sigma_Z + \lambda\,\sigma^2_Z + \int_{|z|<1} {z\,(e^{\lambda z}-1)\gamma(\ud z)}$ \\
		\hline
		$\sigma_Z^{\prime} := \sigma_Z$ \\
		\hline
		$\gamma^{\prime}(\ud z) := e^{\lambda z}\,\gamma(\ud z)$ \\
		\hline
	\end{tabular}
	\caption{Parameter transformations from $\PP$ to $\PP^{\prime}$\\ for the CBITCL process $(X, Z)$.}
	\label{table:newparameterscbitcl}
\end{table}

\begin{proof}
By Lemma \ref{lem:exponentialcompensator}, the process $\exp(\cW - \cK)$ is a local martingale and, by Fatou's lemma, also a supermartingale. Therefore, to prove the martingale property of $(\exp(\cW_t - \cK_t))_{t\in[0,\horizon]}$, it suffices to show that $\EE[\exp(\cW_{\horizon} - \cK_{\horizon})]=1$. More specifically, making use of equations \eqref{eq:processW} and \eqref{eq:exponentialcompensator}, we will prove that
\be\label{eq:expectation}
e^{-\zeta X_0 - \horizon\Psi(\zeta)}\,\EE\bigl[e^{\zeta X_{\horizon} -(\Phi(\zeta) + \zeta\,\lambda\,\rho\,\sigma_X\sigma_Z + \Xi(\lambda))Y_{\horizon} + \lambda Z_{\horizon}}\bigr] = 1.
\ee
Recalling the notation introduced in Section \ref{sec:exponentialmomentscbitcl}, we observe that $\zeta\leq\chi^{(-\Phi(\zeta)-\zeta\,\lambda\,\rho\,\sigma_X\,\sigma_Z-\Xi(\lambda),\lambda)}$. Theorem \ref{thm:lifetimetheoremcbitcl} therefore implies that $\uT^{(\zeta,-\Phi(\zeta)-\zeta\,\lambda\,\rho\,\sigma_X\,\sigma_Z-\Xi(\lambda),\lambda)}=+\infty$, thereby showing that the expectation in \eqref{eq:expectation} is finite.
Moreover, under Assumption \ref{ass:lipschitzcbi}, there exists a unique solution to the extended Riccati system \eqref{eq:extendedcbitclode_0}-\eqref{eq:extendedcbitclode} with $(u_1,u_2,u_3)=(\zeta, -\Phi(\zeta)-\zeta\,\lambda\,\rho\,\sigma_X\,\sigma_Z-\Xi(\lambda), \lambda)$. The solution to \eqref{eq:extendedcbitclode} is given by the constant function $\cV(\cdot,\zeta,-\Phi(\zeta)-\zeta\,\lambda\,\rho\,\sigma_X\,\sigma_Z-\Xi(\lambda),\lambda)=\zeta$, which in turn implies that $\cU(t,\zeta,-\Phi(\zeta)-\zeta\,\lambda\,\rho\,\sigma_X\,\sigma_Z-\Xi(\lambda),\lambda)=t\Psi(\zeta)$, for all $t\geq0$. 
The validity of \eqref{eq:expectation} then follows directly from Lemma \ref{lem:extendedlaplacefouriercbitcl}.
We have thus shown that \eqref{eq:measurechangecbitcl} defines a probability measure $\PP^{\prime}\sim\PP$ with density process $(\exp(\cW_t - \cK_t))_{t\in[0,\horizon]}$.

In order to show that $(X,Z)$ is a CBITCL process under $\PP^{\prime}$, we first express $(\exp(\cW_t-\cK_t))_{t\in[0,\horizon]}$ as a stochastic exponential, making use of \cite[Theorem II.8.10]{JS03} together with \eqref{eq:cbitclsde1}-\eqref{eq:cbitclsde2}:
\begin{align*}
e^{\cW - \cK} &= \cE\biggl( \zeta\,\sigma_X\int_0^{\cdot} {\sqrt{X_s}\,\ud B_s^X} + \lambda\,\sigma_Z\int_0^{\cdot} {\sqrt{X_s}\,\ud B_s^Z} +  \int_0^{\cdot}\int_0^{+\infty} {(e^{\zeta x}-1)\tildeN_0(\ud s,\ud x)} \biggr)\\
&\quad\times\cE\biggl( \int_0^{\cdot}\int_0^{X_{s-}}\int_0^{+\infty} {(e^{\zeta x}-1)\tildeN_1(\ud s,\ud u,\ud x)} + \int_0^{\cdot}\int_0^{X_{s-}}\int_{\R} {(e^{\lambda z}-1)\tildeN_2(\ud s,\ud u,\ud z)} \biggr).
\end{align*}
By Girsanov's theorem, the processes $B^{\prime,X}=(B_t^{\prime,X})_{t\in[0,\horizon]}$ and $B^{\prime,Z}=(B_t^{\prime,Z})_{t\in[0,\horizon]}$ defined by
\[
B_t^{\prime,X} := B_t^X - (\zeta\,\sigma_X + \lambda\,\rho\,\sigma_Z)\int_0^t {\sqrt{X_s}\,\ud s} 
\quad \text{ and } \quad 
B_t^{\prime,Z} := B_t^Z - (\zeta\,\rho\,\sigma_X + \lambda\,\sigma_Z)\int_0^t {\sqrt{X_s}\,\ud s},
\]
for all $t\in[0,\horizon]$, are Brownian motions under the probability measure $\PP^{\prime}$, with correlation $\rho$.
Again by Girsanov's theorem, under $\PP^{\prime}$ the compensated Poisson random measures associated to $N_0(\ud t, \ud x)$, $N_1(\ud t, \ud u, \ud x)$, and $N_2(\ud t, \ud u, \ud z)$ are respectively given by
\begin{align*}
\tildeN_0^{\prime}(\ud t, \ud x) &:= N_0(\ud t, \ud x) - e^{\zeta x}\nu(\ud x)\,\ud t,\\
\tildeN_1^{\prime}(\ud t, \ud u, \ud x) &:= N_1(\ud t, \ud u, \ud x) - e^{\zeta x}\pi(\ud x)\,\ud u\,\ud t,\\
\tildeN_2^{\prime}(\ud t, \ud u, \ud z) &:= N_2(\ud t, \ud u, \ud z) - e^{\lambda z}\gamma(\ud z)\,\ud u\,\ud t.
\end{align*} 
Therefore, under the probability measure $\PP^{\prime}$, the extended Dawson-Li representation \eqref{eq:cbitclsde1}-\eqref{eq:cbitclsde2} of $(X,Z)$ can be rewritten as follows:
\begin{align*}
X_t &= X_0 + \int_0^t {(\beta^{\prime} - b_X^{\prime}X_s)\ud s} + \sigma_X\int_0^t {\sqrt{X_s}\,\ud B_s^{\prime,X}}\\
&\quad+ \int_0^t \int_0^{+\infty} {x\,N_0(\ud s, \ud x)} + \int_0^t \int_0^{X_{s-}}\int_0^{+\infty} {x\tildeN_1^{\prime}(\ud s, \ud u, \ud x)},\\
Z_t &= b_Z^{\prime}\int_0^t {X_s\,\ud s} + \sigma_Z\int_0^t {\sqrt{X_s}\,\ud B_s^{\prime,Z}} + \int_0^t \int_0^{X_{s-}}\int_{|z| \geq 1} {z\,N_2(\ud s, \ud u, \ud z)}\\
&\quad+ \int_0^t \int_0^{X_{s-}}\int_{|z| < 1} {z\,\tildeN_2^{\prime}(\ud s, \ud u, \ud z)},
\end{align*}
where the parameters $\beta^{\prime}$, $b_X^{\prime}$, $b^{\prime}_Z$, $\sigma_X^{\prime}$, $\sigma_Z^{\prime}$, and $\rho^{\prime}$ are given as in Table \ref{table:newparameterscbitcl}. 
By Theorem  \ref{thm:weakequivalencecbitcl}, under the probability measure $\PP^{\prime}$ the process $(X,Z)$ is therefore equivalent to a CBITCL process. 
Moreover, Assumption \ref{ass:no_common_jumps} remains satisfied under $\PP^{\prime}$, as a consequence of the fact that $\PP^{\prime}\sim\PP$.
\end{proof}

As pointed out at the end of Section \ref{sec:affine}, in financial applications the component $Z$ of a CBITCL process $(X,Z)$ is typically related to the log-price process of an asset. In order to ensure absence of arbitrage, it is useful to have conditions characterizing the martingale property of $\exp(Z)$. To this end, by exploiting the previous results, we can state the following corollary.

\begin{cor}	\label{cor:martingale}
Let $(X, Z)$ be a $\uCBITCL(X_0, \Psi, \Phi, \Xi, \rho)$ and suppose that Assumption \ref{ass:lipschitzcbi} holds. Then, the process $(e^{Z_t})_{t\in[0,\horizon]}$ is a martingale if and only if $1\in\cD_Z$ and $\Xi(1)=0$.
\end{cor}
\begin{proof}
If $1\in\cD_Z$ and $\Xi(1)=0$, by making use of \eqref{eq:processW} with $(\zeta,\lambda)=(0,1)$ together with Lemma \ref{lem:exponentialcompensator} and Theorem \ref{thm:girsanovcbitcl}, we directly obtain that $(e^{Z_t})_{t\in[0,\horizon]}$ is a martingale.
Conversely, if $(e^{Z_t})_{t\in[0,\horizon]}$ is a martingale, then $\EE[e^{Z_{\horizon}}]=1$. By \cite[Theorem 2.14-(a)]{KRM15}, it follows that $1\in\cD_Z$. Moreover, in view of \eqref{eq:characterizationmaximumjointlifetimecbitcl}, we have that $\uT^{(0,0,1)}>\horizon$. By Lemma \ref{lem:extendedlaplacefouriercbitcl}, the martingale property of $(e^{Z_t})_{t\in[0,\horizon]}$ necessarily implies that $\cV(t,0,0,1)=0$, for all $t\in[0,\horizon]$. Since the ODE \eqref{eq:extendedcbitclode} admits a unique solution under Assumption \ref{ass:lipschitzcbi}, it  follows that $\Xi(1)=0$.
\end{proof}

\section{Examples and option pricing applications}	\label{sec:examples}

In this section, we present some examples of CBITCL processes that possess a self-exciting behavior and are particularly appropriate for financial applications. In Section \ref{sec:alphacir}, we analyze from the viewpoint of CBITCL processes the alpha-CIR process recently studied in \cite{JMS17,JMSZ21}. In Section \ref{sec:CGMY}, we investigate the CBITCL process adopted in \cite{FGS22} for the modelling of multi-currency markets with stochastic volatility. Finally, in Section \ref{sec:pricing}, we briefly discuss some general aspects of the use of CBITCL processes for option pricing applications. 

\subsection{Alpha-CIR process and geometric Brownian motion}\label{sec:alphacir}

We say that a process $(X,Z)=((X_t, Z_t))_{t\geq0}$ is an \emph{$\alpha$-CIR-time-changed geometric Brownian motion} if it is given by
\be\label{eq:stochtimechangeequationalphacir}
\left\{\ba
X_t &= X_0 + \int_0^t {(\beta - b\,X_s)\,\ud s} + \sigma\,B^X_{Y_t} + \eta\,L_{Y_t},\\
Z_t &= B^Z_{Y_t} -\frac{1}{2}Y_t,
\ea\right.
\ee
with $Y_t=\int_0^tX_s\,\ud s$, for all $t\geq0$, and where $X_0\geq0$, $\beta\geq0$, $b\geq0$, $\sigma\geq0$, $\eta\geq0$, $B^X$ and $B^Z$ are one-dimensional Brownian motions with correlation $\rho$,  with $\rho\in[-1,1]$, and $L$ is a spectrally positive compensated $\alpha$-stable L\'evy process independent of $(B^X,B^Z)$, with stability parameter $\alpha\in(1,2)$ and L\'evy measure $C_{\alpha}\,z^{-1-\alpha}\ind_{\{z>0\}}\ud z$ where $C_{\alpha}$ is a suitable normalization constant. 

This specification has been recently adopted for stochastic volatility modelling in \cite{JMSZ21}, under the name of {\em alpha-Heston stochastic volatility model}. The process defined in \eqref{eq:stochtimechangeequationalphacir} belongs to the class of CBITCL processes introduced in Definition \ref{def:cbitcl}, with $K_t := \beta t$, $M_t := -b t + \sigma B_t^X + \eta L_t$, and $N_t := B_t^Z -t/2$, for all $t\geq0$. 
The component $X$ is a CBI process with $\nu=0$ and $\pi(\ud z) = \eta^{\alpha}C_{\alpha}z^{-1-\alpha}\ind_{\{z>0\}}\ud z$, with immigration mechanism $\Phi(x)=\beta x$ and branching mechanism
\be\label{eq:alphastablebranching}
\Phi(x) = -b\,x + \frac{1}{2}(\sigma\,x)^2 + C_{\alpha}\,\Gamma(-\alpha)(-\eta\,x)^{\alpha},
\ee
where $\Gamma$ denotes the Gamma function extended to $\R\setminus\mathbb{Z}_-$ (see \cite{Lebedev}). 
Under this specification, we have $\phi = 0$ and $\psi = +\infty$, implying that $\cD_X = (-\infty, 0]$. Since $\int_1^{+\infty}z\pi(\ud z)<+\infty$, Assumption \ref{ass:lipschitzcbi} is satisfied. 
Moreover, Assumption \ref{ass:no_common_jumps} is trivially satisfied.
As a consequence of Theorem \ref{thm:lifetimetheoremcbitcl}, the process $X$ does not admit exponential moments of any order, i.e., $\EE[e^{uX_t}]=+\infty$ for all $u>0$ and $t>0$. This fact will motivate the study of {\em tempered} $\alpha$-stable processes in Section \ref{sec:CGMY}.

\begin{rem}
The process $X$ as introduced in \eqref{eq:stochtimechangeequationalphacir} corresponds to the time change representation of a stable Cox-Ingersoll-Ross process (also named \emph{$\alpha$-CIR process}, see \cite{LM15,JMS17,JMSZ21} and \cite[Section 2.6.2]{phdthesis_szulda}). An $\alpha$-CIR process can be also defined by the following SDE:
\[
X_t = X_0 + \int_0^t {(\beta -  b\,X_s)\ud s} + \sigma\int_0^t {\sqrt{X_s}\,\ud B^X_s} + \eta\int_0^t {\sqrt[\alpha]{X_{s-}}\,\ud L_s}.
\]
By  \cite[Corollary 6.3]{FL10}, this SDE admits a unique strong solution that is a CBI process.
\end{rem}

In \eqref{eq:stochtimechangeequationalphacir}, the component $Z$ is defined as a time-changed Brownian motion with drift. This specification ensures that $\exp(Z)=\mathcal{E}(B_Y)$ is a martingale (see Corollary 4.3). In view of financial applications, $Z$ can therefore represent the discounted log-price process of a risky asset under a risk-neutral probability measure. The associated L\'evy exponent is given by $\Xi(u)=u(u-1)/2$. Using the notation introduced in Section \ref{sec:asymptotic}, we have that
\[
\begin{cases}
\chi(u) = 0,&\text{ for }u\in[0,1],\\
\chi(u) < 0, &\text{ otherwise}.
\end{cases}
\]
Corollary \ref{cor:lifetimeTu} therefore implies that $\uT(u)=+\infty$, for every $u\in[0,1]$, while $\uT(u)=0$, for every $u\notin[0,1]$. In other words, for an $\alpha$-CIR-time-changed geometric Brownian motion it holds that $\EE[e^{uZ_t}]<+\infty$ for all $u\in[0,1]$ and $t>0$, while $\EE[e^{uZ_t}]=+\infty$ for all $u\notin[0,1]$ and $t>0$.
These results are consistent with \cite[Proposition 4.1 and Corollary 4.2]{JMSZ21}, where additional restrictions on the model parameters are required.

\begin{rem}
As explained in part (2) of Remark \ref{rem:expmoments}, these results on the finiteness of exponential moments of $Z$ can be used to study the behavior of the implied volatility smile. Let us denote by $\sigma(T,k)$ the implied volatility of a European Call option written on an asset with price process $\exp(Z)$, with maturity $T$ and strike $\exp(k)$.
By applying \cite[Theorems 3.2 and 3.4]{Lee04}, we can deduce that the asymptotic behavior of $\sigma(T,k)$ at extreme strikes is explicitly described as follows:
\[
\underset{k\to\pm\infty}\limsup\,\frac{\sigma^2(T, k)}{|k|} = \frac{2}{T}, \qquad \text{ for all } T > 0.
\]
\end{rem}

Concerning the long-term behavior of the process $Z$, if we assume that $b>0$, then by applying Proposition \ref{prop:longtermbehavior} we obtain
\[
\frac{1}{t}\log\EE\bigl[e^{uZ_t}\bigr] \underset{t\to+\infty}{\longrightarrow} \beta\,\xi(u), \qquad \text{for $u\in[0,1]$.}
\]
If we assume in addition that $\rho \leq 0$ (which is consistent with the so-called {\em leverage effect} in stochastic volatility models, see for instance \cite{CW04}), then the quantity $\xi(u)$ can be explicitly computed as follows, making use of equation \eqref{eq:equationXi}:
\[
\xi(u) = \Phi_u^{-1}\biggl(\frac{u(1-u)}{2}\biggr),
\]
where $\Phi_u^{-1}$ denotes the inverse of the function $x\mapsto\Phi_u(x):= \Phi(x) + \rho\,\sigma\,u\,x$, with $\Phi$ given by \eqref{eq:alphastablebranching}, which can be easily seen to be a bijection from $\R_-$ to $\R_+$, for every $u\in[0,1]$.


\subsection{Tempered $\alpha$-stable CBI process and CGMY process}
\label{sec:CGMY}

The CBITCL process considered in the previous subsection has the drawback that its CBI component does not possess exponential moments. In finance applications, the existence of exponential moments often represents an essential requirement. For this reason, we now present a CBITCL process that enjoys good integrability properties. The example considered in this subsection relies on tempered $\alpha$-stable CBI processes, as introduced in \cite{FGS21} in interest rate modelling (see also \cite[Section 2.7]{phdthesis_szulda}). 

We recall from \cite{FGS21} that a CBI process $X=(X_t)_{t\geq0}$ is said to be {\em tempered $\alpha$-stable} if the L\'evy measures appearing in \eqref{eq:immigrationmechanism} and \eqref{eq:branchingmechanism} are respectively given by 
\[
\nu = 0
\qquad\text{ and }\qquad
\pi(\ud z) = C_{\alpha}\,z^{-1-\alpha}\,e^{-\theta z}\ind_{\{z>0\}}\ud z,
\]
where $\theta > 0$, $\alpha \in (1,2)$ and $C_{\alpha} > 0$ is a normalization constant. 
Under this specification, the immigration mechanism reduces to $\Psi(u)=\beta u$, while the branching mechanism can be explicitly computed as
\be\label{eq:temperedstablebranching}
\Phi(u) = -bu + \frac{1}{2}(\sigma u)^2 + C_{\alpha}\,\Gamma(-\alpha)\bigl( (\theta - u)^{\alpha} - \theta^{\alpha} + \alpha\theta^{\alpha- 1} u\bigr), 
\qquad \text{ for all } u \leq \theta.
\ee
It can be easily verified that $\cD_X = (-\infty, \theta]$ and Assumption \ref{ass:lipschitzcbi} is satisfied (see also \cite[Lemma 2.19]{phdthesis_szulda}). The existence of exponential moments of $X$ can be characterized by relying on Corollary \ref{cor:lifetimetheoremcbitcl} (noting that, in the present case, $\phi=\theta$ and $\psi=+\infty$). Indeed, taking $u_2=u_3=0$ in Corollary \ref{cor:lifetimetheoremcbitcl}, it follows that $\EE[e^{uX_T}] < +\infty$ holds for all $u\leq \theta$ and $T > 0$ if and only if $\Phi(\theta) \leq 0$, namely, if and only if 
\be\label{eq:momentcondition}
b \geq \frac{\sigma^2}{2}\theta + C_{\alpha}\,\Gamma(-\alpha)\,\theta^{\alpha-1}(\alpha - 1).
\ee
We have therefore shown that, in the case of tempered $\alpha$-stable CBI processes, the existence of exponential moments amounts to a simple condition on the parameters characterizing the process.

\begin{rem}
As an example of application of Theorem \ref{thm:girsanovcbitcl}, we show that tempered $\alpha$-stable CBI processes can be easily constructed from non-tempered $\alpha$-stable CBI processes by means of an equivalent change of probability of the type \eqref{eq:measurechangecbitcl}. Let $X$ be an $\alpha$-CIR process, as in Section \ref{sec:alphacir}, and define the process $\cW$ as in \eqref{eq:processW} with $\zeta=-\theta$ and $\lambda=0$,  so that $\cW=\theta(X_0-X)$, with $\theta>0$. By Theorem \ref{thm:girsanovcbitcl}, one can construct a probability measure $\PP^{\prime}\sim\PP$ with density as in \eqref{eq:measurechangecbitcl} such that under $\PP'$ the L\'evy measures of $X$ are given by $\nu^{\prime}=0$ and $\pi^{\prime}(\ud z)=e^{-\theta z}\eta^{\alpha}C_{\alpha}z^{-1-\alpha}\ind_{\{z>0\}}\ud z$. We have therefore obtained that the process $X$ is a tempered $\alpha$-stable CBI process under $\PP^{\prime}$. An analogous change of measure technique has been also employed in \cite[Proposition 4.1]{JMS17}.
\end{rem}

A CBITCL process $(X,Z)$ based on a tempered $\alpha$-stable CBI process can be constructed as follows (see also \cite[Definition 2.18]{phdthesis_szulda}):
\be\label{eq:stochtimechangeequationtemperedstable}
\left\{\ba
X_t &= X_0 + \int_0^t {(\beta - b\,X_s)\ud s} + \sigma\,B_{Y_t} + L_{Y_t}^X,\\
Z_t &= L_{Y_t}^Z,
\ea\right.
\ee
with $Y_t=\int_0^tX_s\ud s$, for all $t\geq0$, where $X_0\geq0$, $\beta\geq0$, $b\in\R$, $\sigma\geq0$, $B$ is a standard Brownian motion independent of $L^X$, which is a spectrally positive  tempered $\alpha$-stable compensated L\'evy process with L\'evy measure $C_{\alpha}\,z^{-1-\alpha}\,e^{-\theta z}\ind_{\{z>0\}}\ud z$, where $\alpha \in (1,2)$, $\theta > 0$,and $C_{\alpha} > 0$ is a suitable normalization constant, and $L^Z$ is a CGMY process (see \cite{CGMY03}) independent of both $B$ and $L^X$.
We recall that $L^Z$ is a CGMY process if its L\'evy measure $\gamma$ is of the form
\be	\label{eq:CGMY_measure}
\gamma(\ud z) = C_Y\bigl(z^{-1-Y}e^{-Mz}\ind_{\{z>0\}} + |z|^{-1-Y}e^{-G|z|}\ind_{\{z<0\}}\bigr)\ud z, 
\ee
where the normalization constant can be chosen as $C_Y = 1/\Gamma(-Y)$. The parameters $G > 0$ and $M > 0$ temper the downward and the upward jumps, respectively, while the parameter $Y \in (1,2)$ determines the local behavior of $L^Z$, similarly to $\alpha$ above. The L\'evy exponent $\Xi$ associated to a CGMY process $L^Z$ with L\'evy triplet $(0, 0, \gamma\bigr)$ is given by
\[
\Xi(u) = \int_{\R} (e^{zu} - 1 - z u)\gamma_Z(\ud z), 
\qquad \text{ for all } u \in \im\R.
\]
It can be easily checked that $\cD_Z = [-G, M]$ and the L\'evy exponent $\Xi$ takes the explicit form
\[
\Xi(u) = (M - u)^Y - M^Y + (G + u)^Y - G^Y + u\,Y\,(M^{Y - 1} - G^{Y - 1}),
\qquad\text{ for all } u\in[-G,M].
\]

For simplicity of presentation, let us assume that $G=M$. In this case, $\Xi:[-M,M]\to\R_+$ is a convex function with minimum $\Xi(0)=0$ and maximum $\Xi(-M)=\Xi(M)=2M^Y(2^{Y-1}-1)$.
If $2M^Y(1-2^{Y-1})\geq\Phi(\theta)$, which can be rewritten in the form
\be\label{eq:conditionM}
M \leq \biggl(\frac{\Phi(\theta)}{2\bigl(1-2^{Y-1}\bigr)}\biggr)^{1/Y},
\ee
we have that $\chi(u) = \theta \geq 0$ for every $u \in [-M, M]$, using the notation introduced in Section \ref{sec:asymptotic}.
Corollary \ref{cor:lifetimeTu} then implies that $\uT(u)=+\infty$ for every $u\in[-M,M]$, while $\uT(u)=0$ for every $u\notin[-M,M]$. 
Similarly as in Section \ref{sec:alphacir}, these results on the finiteness of exponential moments of $Z$ can be used to characterize the tail behavior of the implied volatility smile. Indeed, under condition \eqref{eq:conditionM} and assuming $M>1$, an application of \cite[Theorems 3.2 and 3.4]{Lee04} yields that
\[\begin{aligned}
\underset{k\to-\infty}\limsup\,\frac{\sigma^2\bigl(T, k\bigr)}{|k|} &= \frac{2}{T}\bigl(1 - 2(\sqrt{M^2 + M}-M)\bigr),\\ \underset{k\to+\infty}\limsup\,\,\frac{\sigma^2\bigl(T, k\bigr)}{k} &= -\frac{2}{T}\bigl(1 + 2(\sqrt{M^2 - M}-M)\bigr),
\end{aligned}
\qquad\qquad\text{ for all }T>0.
\]

The long-term behavior of the process $Z$ can be determined by relying on Proposition \ref{prop:longtermbehavior}. Under conditions \eqref{eq:momentcondition} and \eqref{eq:conditionM}, we have that
\[
\frac{1}{t}\log\EE\bigl[e^{uZ_t}\bigr] \underset{t\to+\infty}{\longrightarrow} \beta\,\xi(u), \qquad \text{for all } u\in[-M,M],
\]
where $\xi(u)$ is defined in the statement of Proposition \ref{prop:longtermbehavior}. 

The quantity $\xi(u)$ can be explicitly determined under the following additional condition:
\be\label{eq:improvedmomentcondition}
b \geq \sigma^2\theta + C_{\alpha}\Gamma(-\alpha)\,\theta^{\alpha-1}\alpha.
\ee
Condition \eqref{eq:improvedmomentcondition} is stronger than condition \eqref{eq:momentcondition} and, together with the convexity and continuity of $\Phi$, it implies that $\Phi$ is decreasing on $(-\infty, \theta]$. In this case, making use of \eqref{eq:equationXi}, we have that
\[
\xi(u) = \Phi^{-1}\bigl(2M^Y - (M+u)^Y - (M-u)^Y\bigr),
\]
where $\Phi^{-1}$ denotes the inverse function of the branching mechanism $\Phi$ given by \eqref{eq:temperedstablebranching}, which under condition \eqref{eq:improvedmomentcondition} is a bijection from $(-\infty, \theta]$ to $[\Phi(\theta), +\infty)$.

\begin{rem}\label{remFGS22}
(1) The present specification does not necessarily guarantee the martingale property of the process $\exp(Z)$ (compare with Corollary \ref{cor:martingale}). Therefore, in view of financial applications and similarly as in \cite{FGS22}, the CBITCL process $(X,Z)$ may be used to model the discounted price process $S=(S_t)_{t\geq0}$ of a risky asset as follows:
\begin{align}\label{eq:FGS22model}
\log S_t := \lambda Z_t + \zeta(X_t-X_0) - \cK_t,
\qquad\text{ for all } t\geq0,
\end{align}
where $\cK$ denotes the exponential compensator (see Lemma \ref{lem:exponentialcompensator}), with $\zeta\leq\theta$ and $\lambda\in[-G,M]$. Under this specification, the CBI process $X$ plays the role of stochastic volatility, while the parameter $\zeta$ determines the correlation between the log-price process and its volatility. A direct application of Theorem \ref{thm:girsanovcbitcl} yields that $S$ is a martingale, thereby ensuring absence of arbitrage.
Moreover, since Assumption \ref{ass:lipschitzcbi} is satisfied, the existence of moments $\EE[S^u_t]$ can be characterized by Theorem \ref{thm:lifetimetheoremcbitcl}.

(2)
As a direct consequence of Theorem \ref{thm:girsanovcbitcl}, the class of CBITCL processes considered in this subsection is stable with respect to equivalent changes of probability of the form considered in Section \ref{sec:girsanovcbitcl}. This property has been used to construct risk-neutral measures that preserve the structure of the model in \cite{FGS22}, where CBITCL processes have been applied to the modelling of multi-currency markets in the presence of stochastic volatility and self-exciting jumps.
\end{rem}

\subsection{Applications to option pricing}\label{sec:pricing}

For simplicity of presentation, we assume that the discounted price process of a risky asset is given by \eqref{eq:FGS22model}, with respect to a CBITCL process $(X,Z)$ and with $(\zeta,\lambda)\in\cD_X\times\cD_Z$. In view of Theorem \ref{thm:girsanovcbitcl}, this implies the martingale property of the price process $S$, thereby ensuring absence of arbitrage. In this specification, the process $X$ represents stochastic volatility, possibly correlated and with common jumps with the asset returns.

For illustrative purposes, let us consider the problem of pricing a European call option, with payoff $(S_T-K)^+$ delivered at maturity $T$. To solve this pricing problem, we will rely on Fourier techniques, which are viable in the present setting due to the analytical tractability of CBITCL processes.
Indeed, the characteristic function $\varphi_t(u):=\mathbb{E}[e^{\im u \log S_t}]$, for $u\in\mathbb{R}$ and $t\geq0$, is available in semi-closed form, up to the solution of the system of Riccati ODEs. The following result is directly obtained as a special case of Lemma \ref{lem:extendedlaplacefouriercbitcl} (see \cite[Lemma 3.9]{FGS22} in the specific case of a tempered $\alpha$-stable CBI process, as considered in Section \ref{sec:CGMY}).

\begin{cor}
For every $u\in\R$ and $t\geq0$, the characteristic function $\varphi_t(u)$ is given by
\[
\varphi_t(u) = \exp\bigl( -\im u \zeta X_0-\im u t\Psi(\zeta) + \cU(t, u_1, u_2, u_3) + \cV(t, u_1, u_2, u_3)X_t \bigr),
\]
where the functions $\cU$ and $\cV$ solve \eqref{eq:cbitclriccati1}-\eqref{eq:cbitclriccati2} 
with
\[
u_1:=\im u \zeta,\qquad
u_2:= -\im u \bigl(\Phi(\zeta) + \zeta\,\lambda\,\rho\,\sigma_X\sigma_Z + \Xi(\lambda)\bigr),\qquad
u_3:= \im u\lambda.
\]
\end{cor}

Starting from the above corollary, the pricing problem can be directly solved by relying on the Fourier-based method of \cite{Lee2004}. This approach requires an extension of the domain of $\varphi_T$ to a suitable subset of $\C$. While precise results in this direction can be obtained from Theorem \ref{thm:lifetimetheoremcbitcl}, for the specific case of a call option one can resort to a simpler argument. Indeed, $\varphi_T$ is obviously well posed in $u=0$ and, due to the martingale property of $S$, also in $u=-\im$. In turn, this implies that $\varphi_T$ can be analytically extended to the set $\{u\in\C : -1<\Im(u)<0\}$. We can then apply \cite[Theorem 5.1]{Lee2004} and deduce the following pricing formula, which is valid for every $-1<\alpha<0$:
\be\label{eq:callOptionLee}
\EE[(S_T-K)^+]=\varphi_T(-\im)+\frac{1}{\pi}\int_{0-\alpha\im}^{\infty-\alpha\im}\Re\left(e^{-\im z \log K}\frac{\varphi_T(z-\im)}{-z(z-\im)}\right)\ud z.
\ee

\begin{rem}
In the context of CBITCL processes, the scope of applicability of this Fourier-based pricing technique is not limited to European call options. Indeed, whenever the Fourier transform of the payoff function is known, one can derive a pricing formula along the lines of \eqref{eq:callOptionLee}. Several examples of equity payoff functions for which the Fourier transform is known are provided in \cite{egp2008}. In addition, several payoff functions related to the volatility/variance of the asset price process can be priced by applying the techniques developed in \cite{kmkv2011}, \cite{zk2014} and \cite{dfgg2015}, which are all viable in the present context of CBITCL processes. We also mention \cite{hz2010} and \cite{cfgg2016}, where payoffs depending on several assets are considered, most notably basket options.
\end{rem}

\bibliographystyle{alpha}
\bibliography{biblio_CBITCL}

\end{document}